\documentclass[12pt,oneside]{article}
\usepackage{amssymb,amsmath}
\usepackage{amsfonts}
\usepackage{mathrsfs}
\usepackage{latexsym}
\usepackage{amssymb}
\usepackage{amsthm}
\usepackage{indentfirst}
\begin{document}

\def\abstractname{\bf Abstract}
\def\dfrac{\displaystyle\frac}
\def\dint{\displaystyle\int}
\def\vec#1{\overset\rightarrow{#1}}
\let\oldsection\section
\renewcommand\section{\setcounter{equation}{0}\oldsection}
\renewcommand\thesection{\arabic{section}}
\renewcommand\theequation{\thesection.\arabic{equation}}
\newtheorem{theorem}{\indent Theorem}[section]
\newtheorem{lemma}{\indent Lemma}[section]
\newtheorem{proposition}{\indent Proposition}[section]
\newtheorem{definition}{\indent Definition}[section]
\newtheorem{remark}{\indent Remark}[section]
\newtheorem{corollary}{\indent Corollary}[section]
\renewcommand{\proofname}{\indent\it\bfseries Proof.}

\title{\LARGE\bf
Global Strong Solutions to Incompressible Nematic Liquid Crystal Flow
\\
\author{Jinkai Li$^1$
\thanks{Corresponding author. Email: {\it jkli@math.cuhk.edu.hk}}
\\
{\it \small $^1$  The Institute of Mathematical Sciences, The Chinese University of Hong Kong, Hong Kong}
 } }

\maketitle

\begin{abstract}
In this paper, we consider the Dirichlet problem of inhomogeneous incompressible nematic liquid crystal equations in bounded smooth domains of two or three dimensions. We prove the global existence and uniqueness of strong solutions with initial data being of small norm but
allowed to have vacuum. More precisely, for two dimensional case, we only require that the basic energy $\|\sqrt{\rho_0}u_0\|_{L^2}^2+\|\nabla d_0\|_{L^2}^2$ is small, while for three dimensional case, we ask for the smallness of the production of the basic energy and the quantity $\|\nabla u_0\|_{L^2}^2+\|\nabla^2d_0\|_{L^2}^2$. Our efforts mainly center on the establishment of the time independent a priori estimate on local strong solutions. Taking
advantage of such a priori estimate, we extend the local strong solution to the whole time, obtaining the global strong solution.
\end{abstract}


\section{Introduction}\label{sec1}

We consider the following hydrodynamic system modeling the flow of
nematic liquid crystal materials
\begin{align}
&\rho_t+\textmd{div}(\rho u)=0,\label{1.1} \\
&\rho(u_t+(u\cdot\nabla)u)-\nu\Delta u+\nabla
p=-\lambda\textmd{div}(\nabla d\odot\nabla d),\label{1.2}\\
&\textmd{div}u=0,\label{1.3}\\
&d_t+(u\cdot\nabla )d=\gamma(\Delta d+|\nabla d|^2d)\label{1.4}
\end{align}
in $\Omega\times(0,\infty)$, where $\Omega$ is a bounded domain with
smooth boundary in $\mathbb R^N$ $(N=2,3)$. Here
$u:\Omega\times(0,\infty)\rightarrow\mathbb R^N$ represents the
velocity field of the flow, $d:\Omega\times(0,\infty)\rightarrow
\mathcal{S}^2$, the unit sphere in $\mathbb R^3$, represents the
macroscopic molecular orientation of the liquid crystal material,
$\rho:\Omega\times(0,\infty)\rightarrow[0,\infty)$ and
$p:\Omega\times(0,\infty)\rightarrow\mathbb R$ are scalar functions,
respectively, denoting the density of the fluid and the pressure
arising from the usual assumption of incompressibility
$\textmd{div}u=0$. The positive constants $\nu,\lambda$ and $\gamma$
represent viscosity, the competition between kinetic energy and
potential energy, and microscopic elastic relaxation time or the
Dehorah number for the molecular orientation field, respectively.
The symbol $\nabla d\odot\nabla d$, which exhibits the property of
the anisotropy of the material, denotes the $N\times N$ matrix whose
$(i,j)$-th entry is given by $\partial_i d\cdot\partial_jd$ for
$1\leq i,j\leq N$.

Noticing that
$$
\textmd{div}(\nabla d\odot\nabla d)=\Delta d\cdot\nabla d+\nabla\left(\frac{|\nabla d|^2}{2}\right),
$$
one can rewrite equation \ref{1.2} as
\begin{equation}\label{1.2-1}
\rho(u_t+(u\cdot\nabla)u)-\nu\Delta u+\nabla
\left(p+\frac{\lambda}{2}|\nabla d|^2\right)=-\lambda\Delta d\cdot\nabla d.
\end{equation}

System (\ref{1.1})--(\ref{1.4}) is a simplified version of the
Ericksen-Leslie model, which reduces to the Ossen-Frank model in the
static case, for the hydrodynamics of nematic liquid crystals
developed by Ericksen \cite{E1}, \cite{E2} and Leslie \cite{LE} in
the 1960's. Both the full Ericksen-Leslie model and the simplified
version are the macroscopic continuum description of the time
evolution of the materials under the influence of both the flow
velocity field $u$ and the microscopic orientation configurations
$d$ of rod-like liquid crystals. A brief account of the
Ericksen-Leslie theory and the derivations of several approximate
systems can be found in the appendix of \cite{LL1}. For more details
of physics, we refer the readers to the two books of Gennes-Prost
\cite{GP} and Chandrasekhar \cite{CH}. Though the above system is a
simplified version of the full Ericksen-Leslie system, it still
remains the most important mathematical structures as well as most
of the essential difficulties of the original Ericksen-Leslie
system.

In the homogeneous case, i.e. $\rho\equiv C$, Lin-Lin \cite{LL1,LL2}
initiated the mathematical analysis of (\ref{1.2})--(\ref{1.4}) in
the 1990's. More precisely, they considered in \cite{LL1} the Leslie
system of variable length, that is replacing $|\nabla d|^2d$ by the
Ginzburg-Landau type approximation term
$\frac{1-|d|^2}{\varepsilon^2}d$ to relax the nonlinear constraint
$|d|=1$, and proved the existence of global weak solutions in
dimension two or three. They also obtain the unique existence of
global classical solutions in dimension two or in dimension three
with $\nu$ large enough. Furthermore, they proved in \cite{LL2} the
partial regularity theorem for suitable weak solutions, similar to
the classical theorem by Caffarelli-Kohn-Nirenberg \cite{CKN} for
the Navier-Stokes equation. A preliminary analysis of the asymptotic
behavior of global classical solutions was also given in \cite{LL1}.
More precise asymptotic behavior of classical solutions can be found
in Wu \cite{Wu}, in particular, he provided an estimate on the
convergence rate in dimension two. However, as pointed out in
\cite{LL1,LL2}, both the estimates and arguments in these two papers
depend on $\varepsilon$, and it's a challenging problem to study the
convergence as $\varepsilon$ tends to zero. The two dimensional case
is comparatively easier, in fact Hong \cite{Hong} obtains the convergence as $\varepsilon$ goes to zero up to the first singular time. Such convergence
problem in dimension three is still open. Alternatively, one can
establish the existence of global weak solutions directly to the
system (\ref{1.2})--(\ref{1.4}) but for the Ginzburg-Landau
approximate system. Recently, Lin-Lin-Wang \cite{LLW}
proved the global existence of weak solution to the system
(\ref{1.2})--(\ref{1.4}) in dimension two, and obtained the
regularity and asymptotic behavior of the weak solutions they
established. The uniqueness of such weak solution was later proven
in \cite{LW}. For three dimensional case, the local or global
existence of weak solutions is still unclear in the present.

In the non-homogeneous case, i.e. the density dependent case, the global existence of weak solutions to the system (\ref{1.1})--(\ref{1.4}) with
$|\nabla d|^2d$ being replaced by $\frac{1-|d|^2}{\varepsilon^2}d$, the
Ginzburg-Laudan type approximation term, is established in \cite{Tan1,Tan2} and
\cite{LiuX} for
each $\varepsilon>0$. They cannot get the uniform estimates with
respect to $\varepsilon>0$, and therefore cannot take the limitation
$\varepsilon\rightarrow0$. It's also a challenging problem to study
the convergence as $\varepsilon$ tends to zero for the
non-homogeneous case. If the initial data gains more regularities, one can expect to prove the existence of more regular solutions. In fact, Wen and
Ding \cite{WD} obtain the local existence and uniqueness of the
strong solutions to system (\ref{1.1})--(\ref{1.4}) with initial
density being allowed to have vacuum. If the initial data is small or satisfies some geometric condition, one can obtain the global existence results: global existence of strong solutions in three dimensions with small initial data are obtained by Li and Wang in \cite{LIXL1} for constant density case, Li and Wang in \cite{LIX2} for nonconstant but positive density case, and Ding, Huang and Xia in \cite{DHXIA} for nonnegative density case; global existence of strong and weak solutions in two dimensions with large initial data is obtained by Li \cite{LIJ} under the condition that the third component of the initial direction filed is away from zero.

In the present paper, we consider the global existence of strong solutions to the liquid crystal equations. More precisely, we establish the
global existence of strong solutions to the non-homogeneous system
(\ref{1.1})--(\ref{1.4}), coupled with the following initial and
boundary conditions:
\begin{align}
&(\rho,u,d)|_{t=0}=(\rho_0,u_0,d_0), \quad\mbox{with }
|d_0|=1,\quad\textmd{div}u_0=0,\quad\mbox{ and
}u_0|_{\partial\Omega}=0,\label{1.5} \\
&u(x,t)=0,\qquad d(x,t)=d_0^*,\quad\mbox{for
}(x,t)\in\partial\Omega\times(0,\infty),\label{1.6}
\end{align}
where $d_0^*$ is a given unit constant vector and $\rho_0(x)$ a
given nonnegative function being allowed to vanish on some subset of
$\Omega$. Compared with the approximation problem, the term $|\nabla d|^2d$
in (\ref{1.4}) brings us some new difficulties, for example, one can not obtain the a priori $L^2$ estimates on $\Delta d$ from the basic energy identity. System
(\ref{1.1})--(\ref{1.4}) can be viewed as Navier-Stokes equations
coupling the heat flow of harmonic maps. Since the strong solution
of a harmonic map can be blow up in finite time \cite{Chang}, one
cannot expect that (\ref{1.1})--(\ref{1.6}) have a global strong
solution with general initial data. Therefore, we consider the case
that the initial data is of small norm but the initial density
$\rho_0$ is allowed to have vacuum.

Throughout this paper, for any $1\leq p\leq \infty$, we denote by
$\|u\|_p$ the $L^p$ norm of $u$ for any $u\in L^p(\Omega)$. Using
this notation, we can state the main result of this paper as
follows:

\begin{theorem}\label{theorem1}
Let $\Omega$ be a bounded smooth domain in $\mathbb{R}^N$ $(N=2,3)$. Assume that $\rho_0\in H^1(\Omega)\cap
L^\infty(\Omega)$, $0\leq\rho_0(x)\leq\bar\rho$ in $\Omega$, $u_0\in
H^2(\Omega)\cap H^1_0(\Omega)$ with $\emph{\textmd{div}}u_0=0$ in
$\Omega$, $d_0\in H^3(\Omega)$ with $|d_0|=1$ and
$d_0=d_0^*$ on $\partial\Omega$, for some positive constant $\bar\rho$ and a constant unit vector $d_0^*$,
and the following compatible
condition is valid
$$
-\nu\Delta u_0-\nabla p_0-\lambda\emph{\textmd{div}}(\nabla
d_0\odot\nabla d_0)=\sqrt\rho_0g_0
$$
in $\Omega$ for $(p_0,g_0)\in H^1(\Omega)\times L^2(\Omega)$.

Then there is a positive constant $\varepsilon_0$ depending only on $\nu$, $\lambda$, $\gamma$, $\bar\rho$ and $\Omega$, such that if
\begin{eqnarray*}
&&\|\sqrt{\rho_0}u_0\|_2^2+\|\nabla d_0\|_2^2<\varepsilon_0 \mbox{ for }N=2,\\
&&\mbox{ or }(\|\sqrt{\rho_0}u_0\|_2^2+\|\nabla d_0\|_2^2)(\|\nabla u_0\|_2^2+\|\nabla^2d_0\|_2^2)<\varepsilon_0\mbox{ for }N=3,
\end{eqnarray*}
then initial boundary value problem (\ref{1.1})--(\ref{1.6}) has a
unique global strong solution $(\rho,u,p,d)$ satisfying
\begin{align*}
&\rho\in L_{loc}^\infty([0,\infty);H^1(\Omega))\cap
L^{\infty}(0,\infty;L^\infty(\Omega)),\qquad\rho_t\in
L_{loc}^\infty([0,\infty);L^2(\Omega)),\\
&u\in L^\infty(0,\infty;H^2(\Omega)\cap H_0^1(\Omega))\cap
L^2(0,\infty;W^{2,6}(\Omega)),\\
&u_t\in L^2(0,\infty;H_0^1(\Omega)),\qquad \sqrt\rho u_t\in
L^\infty(0,\infty;L^2(\Omega)),\\
&p\in L^\infty(0,\infty;H^1(\Omega))\cap L^2(0,T;W^{1,6}(\Omega)),\\
&d\in W^{4,2}_2(Q)\cap L^\infty(0,\infty;H^3(\Omega)),\qquad d_t\in
L^\infty(0,\infty;H_0^1(\Omega)),\qquad|d|=1,
\end{align*}
where $Q=\Omega\times(0,\infty)$.
\end{theorem}

We now comment on the analysis of this paper. Since the local existence of strong solutions to system (\ref{1.1})--(\ref{1.6}) has been proven in \cite{WD},
to establish the global existence result, we only need to extend the local
solution to the global one. For this aim, recalling that vacuum is allowed in our paper, we need to establish some a priori estimates on local strong solutions,
which is independent of the existence time interval and the lower bound of the density. The first key estimate of this paper is the estimate on $E_1(t)$ (see Lemma \ref{lemL.1} for the definition), which controls the $L^\infty(0,T; H^1)$ norm of the velocity $u$ and the $L^\infty(0, T; H^2)$ norm of the direction field $d$. Via energy estimates, we obtain a polynomial inequality of $E_1(t)$ with small coefficients (see Lemma \ref{lemL.1} for the detail), with an additional term involving $\int_0^t\|u\|_\infty\|\nabla u\|_2^2ds$ if $N=2$. This additional term results from the assumption imposed on the initial data that only the initial basic energy is small for $N=2$, and it disappears if we impose the same assumption for $N=2$ to that for $N=3$. On account of such polynomial inequality on $E_1(t)$, one can use continuity argument to derive the bound of $E_1(t)$ if $N=3$, while for $N=2$, we can employ a logarithmic type Sobolev inequality to obtain the bound of $E_1(t)$. As long as we obtain the estimate on $E_1(t)$, the next step is to do higher order estimates, i.e. the estimates on $E_2(t)$, which controls the $L^\infty(0, T; H^2)$ norm of the velocity $u$ and the $L^\infty(0, T; H^3)$ norm of the direction fields $d$. Similar to the situation encountered before, the arguments are different for $N=2$ and those for $N=3$, and the cause is still the different assumptions imposed on the initial data for $N=2$ and $N=3$. In fact, if we use the same approach to that used for $N=3$ to the deal with the case $N=2$, we will encounter a term $\int_0^t\|\nabla d|_2^2\|u_t\|_2^2ds$, which can not be controlled in terms of $E_1(t)$ under the assumption that only the basic energy of the initial data is small. After obtaining the higher order estimates on local strong solutions, we use the standard approximation approach to establish the global existence of strong solutions, and thus finish the proof.

The rest of this paper is arranged as follows: in Section \ref{sec2}, we state some preliminary lemmas which will be used in the next two sections;
in Section \ref{sec3}, we do the a priori estimates on the local strong solutions, including both the basic energy estimates and the estimates on higher derivatives of $u$ and $d$ independent of the existence time interval and the lower bound of density; in Section \ref{sec4}, taking advantage of the a priori estimates established in Section \ref{sec3}, we prove the global existence and uniqueness of strong solutions by using the standard approximation approach.

Since the exact values of $\nu$, $\lambda$ and $\gamma$ don't play a role, we henceforth assume
$$
\nu=\lambda=\gamma=1
$$
thoughout this paper. We denote
$$
C_0=\|\sqrt{\rho_0}u_0\|_2^2+\|\nabla d_0\|_2^2
$$
the basic energy of the initial data in the rest of this paper.

\section{Preliminaries}\label{sec2}

In this section, we give some useful lemmas which will be used in the rest of this paper.

The following result is quite standard (as a matter of fact, it's a straightforward generalization of the one presented in \cite{Lad}).

\begin{lemma}\label{est on transport} Let $\Omega$ be a Lipschitz domain of $\mathbb R^N$ and $v\in L^1(0,T;Lip)$ be a solenoidal vector-field such that
$v\cdot n=0$ on $\partial\Omega$, where $n$ denotes the outward normal vector on $\partial\Omega$. Let $\rho_0\in W^{1,q}(\Omega)$ with $q\in[1,\infty]$. Then
equation
\begin{equation*}
\left\{
\begin{array}{l}
\rho_t+\emph{\textmd{div}}(\rho v)=0\qquad\mbox{ in }\Omega,\\
\rho|_{t=0}=\rho_0\qquad\mbox{ in }\Omega
\end{array}
\right.
\end{equation*}
has a unique solution in $L^\infty(0,T;W^{1,\infty}(\Omega))\cap C([0,T];\cap_{1\leq r<\infty}W^{1,r}(\Omega))$ if $q=\infty$ and in $C([0,T];W^{1,q}(\Omega))$ if $1\leq q<\infty$.

Besides, the following estimate holds true
$$
\|\rho(t)\|_{W^{1,q}(\Omega)}\leq e^{\int_0^t\|\nabla v(\tau)\|_\infty d\tau}\|\rho_0\|_{W^{1,q}(\Omega)}
$$
for any $t\in[0,T]$. If in addition $\rho$ belongs to $L^p$ for some $p\in[1,\infty]$, then
$$
\|\rho(t)\|_p=\|\rho_0\|_p
$$
for all $t\in[0,T]$. Finally, if $\rho_0(x)\geq\delta$ in $\Omega$ for some positive constant $\delta$, then $\rho(x,t)\geq\delta$ for all $(x,t)\in\Omega\times[0,T]$.
\end{lemma}

We will frequently use the following two lemmas, which state the elliptic $L^q$ estimates on Laplace and Stokes equations.

\begin{lemma} (See \cite{Krylov})\label{est on ellip}
Let $\Omega$ be a $C^{k+2}$ bounded domain in $\mathbb R^N, N\geq2$, $k\geq0$ is an integer.
Then we have
$$
\|u\|_{W^{k+2,q}(\Omega)}\leq C(\|\Delta u\|_{W^{k,q}(\Omega)}+\|u\|_{L^q(\Omega)}+\|g\|_{W^{k+2,q}(\Omega)})
$$
for any $u\in W^{k+2,q}(\Omega)$ and $g\in W^{k+2, q}(\Omega)$ with $u|_{\partial\Omega}=g$, $1<q<\infty$,
where $C$ is a positive constant depending only on $N,q$ and $\Omega$.
\end{lemma}

\begin{lemma} (See \cite{Galdi})\label{est on stok} Let $\Omega$ be a bounded domain in $\mathbb R^N$, $N\geq2$, of class
$C^{m+2}$, $m\geq0$. For any
$$
f\in W^{m,q}(\Omega),\qquad\varphi\in W^{m+2-1/q,q}(\partial\Omega),
$$
$1<q<\infty$, with compatible condition $\int_{\partial\Omega}\varphi\cdot ndS=0,$ there exists one and only one pair $u,p$ such that

(i) $v\in W^{m+2,q}(\Omega)$ and $p\in W^{m+1,q}(\Omega)/\mathbb R$,

(ii) $v,p$ verify the Stokes equation

\begin{equation*}
\left\{
\begin{array}{l}
-\Delta u+\nabla p=f\quad\mbox{ in }\Omega,\\
\emph{\textmd{div }}u=0\quad\mbox{ in }\Omega,\\
u|_{\partial\Omega}=\varphi.
\end{array}
\right.
\end{equation*}
In addition, this solution obeys the inequality
$$
\|v\|_{W^{m+2,q}(\Omega)}+\|p\|_{W^{m+1,q}(\Omega)/{\mathbb R}}\leq C(\|f\|_{W^{m,q}(\Omega)}+\|\varphi\|_{W^{m+2-1/q,q}(\partial\Omega)}),
$$
where $C$ is a positive constant depending only on $N,m,q$ and $\Omega$.

We also need the following local existence result.

\begin{lemma}(See \cite{WD})\label{local existence} Under the conditions stated in Theorem \ref{theorem1}, there is a constant $T_*>0$, such that
for any $T\leq T_*$ system (\ref{1.1})--(\ref{1.6}) has unique solution $(\rho,u,p,d)$ satisfying
\begin{align*}
&\rho\in L^\infty(0,T;H^1(\Omega))\cap
L^{\infty}(0,T;L^\infty(\Omega)),\qquad\rho_t\in
L^\infty(0,T;L^2(\Omega)),\\
&u\in L^\infty(0,T;H^2(\Omega)\cap H_0^1(\Omega))\cap
L^2(0,T;W^{2,6}(\Omega)),\\
&u_t\in L^2(0,T;H_0^1(\Omega)),\qquad \sqrt\rho u_t\in
L^\infty(0,T;L^2(\Omega)),\\
&p\in L^\infty(0,T;H^1(\Omega))\cap L^2(0,T;W^{1,6}(\Omega)),\\
&d\in W^{4,2}_2(Q_T)\cap L^\infty(0,T;H^3(\Omega)),\qquad d_t\in
L^\infty(0,T;H_0^1(\Omega)),\qquad|d|=1,
\end{align*}
where $Q_T=\Omega\times(0,T)$.
\end{lemma}

\end{lemma}

\section{A priori estimates}\label{sec3}
\allowdisplaybreaks
In this section, we concern on the energy estimates on strong
solutions. Let $T>0$ and $(\rho,u,p,d)$ be a strong solution to
(\ref{1.1})--(\ref{1.6}) in $\Omega\times[0,T)$ stated in Lemma \ref{local existence}.
By Lemma \ref{est on transport}, we have
$$
0\leq\rho(x,t)\leq\bar\rho
$$
on $Q_T$. Recalling that $d|_{\partial\Omega}=d_0^*$, then Sobolev embedding theorem implies
$$
\|\nabla^md\|_6^2=\|\nabla^m(d-d_0^*)\|_6^2\leq\|d-d_0^*\|_{H^{m+1}}^2\leq\|\nabla^{m+1}d\|_2^2
$$
for any integer $m\geq 1$ if $N=3$. We will frequently used these facts without any further mentions later.

\begin{lemma}\label{lemL.0} Let $(\rho,u,p,d)$ be a
strong solution to (\ref{1.1})--(\ref{1.6}) in $Q_T$. Then we have
the following energy estimates
\begin{align*}
\sup_{0\leq t\leq T}(\|\sqrt\rho u\|_2^2+\|\nabla
d\|_2^2)+2\int_0^T(\|\nabla u\|_2^2+\|\Delta d\|_2^2)dt  \leq&
C_0+2\int_0^T\|\nabla d\|_4^4dt,
\end{align*}
and
$$
\sup_{0\leq t\leq T}(\|\sqrt\rho u\|_2^2+\|\nabla
d\|_2^2)+2\int_0^T\|\nabla u\|_2^2dt\leq C_0.
$$
\end{lemma}

\begin{proof}
Multiplying (\ref{1.2}) by $u$, using (\ref{1.1}) and integration by
parts, we obtain
\begin{align*}
&\frac{d}{dt}\int_\Omega\frac{\rho|u|^2}{2}dx+\int_\Omega|\nabla
u|^2dx=-\int_\Omega(\Delta d\cdot\nabla d)\cdot
udx\\
=&-\int_\Omega(u\cdot\nabla ) d\cdot\Delta ddx =-\int_\Omega\Delta
d\cdot(\Delta d+|\nabla d|^2d-d_t)dx,
\end{align*}
from which we get
\begin{align*}
\frac{1}{2}\frac{d}{dt}\int_\Omega&\big(\rho|u|^2+|\nabla d|^2\big)dx+\int_\Omega\big(|\nabla u|^2+|\Delta d|^2\big)dx\\
=&-\int_\Omega|\nabla d|^2\Delta d\cdot ddx=\int_\Omega|\nabla d|^4dx,
\end{align*}
here we have used the fact that $\Delta d\cdot d=-|\nabla d|^2$, which is guaranteed by $0=\Delta|d|^2=2\Delta d\cdot d+2|\nabla d|^2.$
Hence
\begin{align*}
&\sup_{0\leq t\leq T}(\|\sqrt\rho u\|_2^2+\|\nabla d\|_2^2)+2\int_0^T(\|\nabla u\|_2^2+\|\Delta d\|_2^2)dt\nonumber\\
\leq& 2\int_0^T\|\nabla d\|_4^4dt+(\|\sqrt\rho_0u_0\|_2^2+\|\nabla
d_0\|_2^2)\leq C_0+2\int_0^T\|\nabla d\|_4^4dt.\label{1}
\end{align*}
Since $|\nabla d|^2=-\Delta d\cdot d$ and $|d|=1$, it follows
$\|\nabla d\|_4^4\leq\|\Delta d\|_2^2$, and thus we deduce from the
above inequality that
$$
\sup_{0\leq t\leq T}(\|\sqrt\rho u\|_2^2+\|\nabla
d\|_2^2)+2\int_0^T(\|\nabla u\|_2^2+\|\Delta d\|_2^2)dt\leq
C_0+2\int_0^T\|\Delta d\|_2^2dt,
$$
which implies
\begin{equation*}\label{2.2-1}
\sup_{0\leq t\leq T}(\|\sqrt\rho u\|_2^2+\|\nabla
d\|_2^2)+2\int_0^T\|\nabla u\|_2^2dt\leq C_0.
\end{equation*}
The proof is complete.
\end{proof}

\begin{lemma}\label{lemL.1}Let $(\rho,u,p,d)$ be a
strong solution to (\ref{1.1})--(\ref{1.6}) in $Q_T$, and set
$$
E_1(t)=\sup_{0\leq s\leq t}(\|\nabla u\|_2^2+\|\nabla^2d\|_2^2)+\int_0^t(\|\sqrt\rho u_t\|_2^2+\|\nabla^2u\|_2^2+\|\nabla p\|_2^2+\|\nabla^3d\|_2^2)ds.
$$
Then we have the following

(i) If $N=2$, there holds
\begin{equation*}
E_1(t)\leq C(\|\nabla u_0\|_2^2+\|\nabla^2d_0\|_2^2)+CC_0E_1(t)+CC_0^{1/2}\int_0^t\|u\|_\infty^2\|\nabla u\|_2^2dt+C\int_0^t\|\nabla^2d\|_2^2ds.
\end{equation*}

(ii) If $N=3$, there holds
\begin{equation*}
E_1(t) \leq C(\|\nabla u_0\|_2^2+\|\nabla^2d_0\|_2^2)+CC_0E_1(t)^2+C\int_0^t\|\nabla^2d\|_2^2ds.
\end{equation*}
\end{lemma}

\begin{proof}
Multiplying (\ref{1.2}) by $u_t$, integration by parts and using
Young inequality, we get
\begin{align*}
&\frac{d}{dt}\int_\Omega\frac{|\nabla u|^2}{2}dx+\int_\Omega\rho|u_t|^2dx=\int_\Omega\big(\nabla d\odot\nabla d\nabla u_t-\rho(u\cdot\nabla )u\cdot u_t)dx\nonumber\\
=&\frac{d}{dt}\int_\Omega\nabla d\odot \nabla d\nabla udx-\int_\Omega\left((\nabla d_t\odot\nabla d+\nabla d\odot\nabla d_t)\nabla u+\rho(u\cdot\nabla)u\cdot u_t\right)dx,
\end{align*}
and thus
\begin{align*}
&\frac{d}{dt}\int_\Omega\left(\frac{|\nabla u|^2}{2}-\nabla d\odot\nabla d\nabla u\right)dx+\int_\Omega\rho|u_t|^2dx\\
\leq&2\int_\Omega(|\nabla d||\nabla d_t||\nabla u|+\rho|u\cdot\nabla u||u_t|)dx\nonumber\\
\leq&\frac{1}{2}\int_\Omega\rho|u_t|^2dx+C\int_\Omega(|\nabla d||\nabla d_t||\nabla u|+\rho|u\cdot\nabla u|^2)dx,
\end{align*}
which gives
\begin{equation}\label{n3.1}
\frac{d}{dt}\int_\Omega\left(|\nabla u|^2-2\nabla d\odot\nabla d\nabla u\right)dx+\int_\Omega\rho|u_t|^2dx\leq C\int_\Omega(|\nabla d||\nabla d_t||\nabla u|+\rho|u\cdot\nabla u|^2)dx.
\end{equation}
Applying elliptic estimates of Stokes equations to (\ref{1.2}) yields
\begin{equation}
\|\nabla^2u\|_2^2+\|\nabla p\|_2^2\leq C(\|\rho(u_t+u\cdot\nabla u)\|_2^2+\|\nabla d|\nabla^2d|\|_2^2).\label{L.2}
\end{equation}
Combining (\ref{n3.1}) with (\ref{L.2}), there holds
\begin{align}
&\sup_{0\leq s\leq t}\|\nabla u\|_2^2+\int_0^t(\|\sqrt\rho u_t\|_2^2+\|\nabla^2u\|_2^2+\|\nabla p\|_2^2)ds\nonumber\\
\leq&C\|\nabla u_0\|_2^2+C\sup_{0\leq s\leq t}\||\nabla d|^2\nabla u\|_1+\varepsilon\int_0^T\|\nabla d_t\|_2^2ds\nonumber\\
&+\int_0^t\int_\Omega(|\nabla d|^2|\nabla u|^2+\rho|u|^2|\nabla u|^2+|\nabla d|^2|\nabla^2d|^2)dxds.\label{L.3}
\end{align}
Taking the operator $\Delta$ to both sides of equation (\ref{1.4}) and then multiplying the resulting equation by $\Delta d$, we deduce
\begin{align}
&\frac{d}{dt}\int_\Omega\frac{|\Delta d|^2}{2}dx-\int_\Omega\Delta^2d\Delta ddx\nonumber\\
=&\int_\Omega(|\nabla d|^2|\Delta d|^2+2\nabla d:\nabla \Delta dd\Delta d+2|\nabla^2d|^2d\Delta d\nonumber\\
&-(\Delta u\cdot\nabla)d\Delta d-2(\nabla u_i\partial_i\nabla)d\Delta d)dx\nonumber\\
\leq&\varepsilon\int_\Omega(|\nabla\Delta d|^2+|\Delta u|^2)dx+C\int_\Omega(|\nabla d|^2|\nabla^2d|^2+|\nabla u||\nabla^2d|^2)dx.\label{L.4}
\end{align}
In the above we have used the fact that $d\cdot\Delta d=-|\nabla d|^2$ guaranteed by $|d|=1$. Note that $\Delta d|_{\partial\Omega}=|\nabla d|^2d|_{\partial\Omega}$ guaranteed by equation (\ref{1.4}) and the boundary condition (\ref{1.6}). Integration by parts gives
\allowdisplaybreaks\begin{align*}
-\int_\Omega\Delta^2d\Delta ddx=&-\int_{\partial\Omega}\Delta d\frac{\partial}{\partial n}\Delta ddS+\int_\Omega|\nabla\Delta d|^2dx\\
=&\int_{\partial\Omega}|\nabla d|^2d\frac{\partial}{\partial n}\Delta ddS+\int_\Omega|\nabla \Delta d|^2dx\\
=&\int_{\partial\Omega}\left(|\nabla d|^2\frac{\partial}{\partial n}(d\Delta d)-|\nabla d|^2\Delta d\frac{\partial}{\partial n}d\right)dS+\int_\Omega|\nabla\Delta d|^2dx\\
=&-\int_{\partial\Omega}\left(|\nabla d|^2\frac{\partial}{\partial n}|\nabla d|^2+|\nabla d|^2\Delta d\frac{\partial}{\partial n}d\right)dS+\int_\Omega|\nabla\Delta d|^2dx,
\end{align*}
which, combined with (\ref{L.4}), it follows from the trace inequality that
\begin{align*}
&\frac{d}{dt}\int_\Omega\frac{|\Delta d|^2}{2}dx+\int_\Omega|\nabla\Delta d|^2dx\\
\leq&\varepsilon\int_\Omega(|\nabla\Delta d|^2+|\Delta u|^2)dx+C\int_{\partial\Omega}|\nabla d|^3|\nabla^2d|dS\\
&+C\int_\Omega(|\nabla d|^2|\nabla^2d|^2+|\nabla u||\nabla^2d|^2)dx\\
\leq&\varepsilon\int_\Omega(|\nabla\Delta d|^2+|\Delta u|^2)dx+C\||\nabla d|^3|\nabla^2d|\|_{W^{1,1}(\Omega)}\\
&+C\int_\Omega(|\nabla d|^2|\nabla^2d|^2+|\nabla u||\nabla^2d|^2)dx\\
\leq&\varepsilon\int_\Omega(|\nabla\Delta d|^2+|\Delta u|^2)dx+C\int_{\Omega}(|\nabla d|^3|\nabla^2d|+|\nabla d|^2|\nabla^2d|^2+|\nabla d|^3|\nabla^3d|)dx\\
&+C\int_\Omega(|\nabla d|^2|\nabla^2d|^2+|\nabla u||\nabla^2d|^2)dx\\
\leq&\varepsilon\int_\Omega(|\nabla\Delta d|^2+|\Delta u|^2)dx+C\int_\Omega(|\nabla d|^2|\nabla^2d|^2+|\nabla d|^3|\nabla^3d|\\
&+|\nabla u||\nabla^2d|^2+|\nabla d|^3|\nabla^2d|)dx\\
\leq&\varepsilon\int_\Omega(|\nabla^3 d|^2+|\Delta u|^2)dx+C\int_\Omega(|\nabla d|^2|\nabla^2d|^2+|\nabla d|^6\\
&+|\nabla u||\nabla^2d|^2+|\nabla^2d|^2)dx,
\end{align*}
and thus
\begin{align*}
&\sup_{0\leq t\leq T}\|\nabla^2 d\|_2^2+\int_0^T\|\nabla^3d\|_2^2dt\nonumber\\
\leq& C\|\nabla^2d_0\|_2^2+\varepsilon\int_0^T\|\nabla^2u\|_2^2dt+C\int_0^T(|\nabla d|^2|\nabla^2d|^2\\
&+|\nabla u||\nabla^2d|^2+|\nabla d|^6+|\nabla^2d|^2)dxdt.
\end{align*}
Combining (\ref{L.3}) with the above inequality, and using Lemma \ref{lemL.0}, we obtain
\begin{align}
&\sup_{0\leq s\leq t}(\|\nabla u\|_2^2+\|\nabla^2d\|_2^2)+\int_0^t(\|\sqrt\rho u_t\|_2^2+\|\nabla^2u\|_2^2+\|\nabla p\|_2^2+\|\nabla^3d\|_2^2)ds\nonumber\\
\leq&C(\|\nabla u_0\|_2^2+\|\nabla^2d_0\|_2^2)+C\sup_{0\leq s\leq t}\||\nabla d|^2\nabla u\|_1+C\int_0^t\int_\Omega(|\nabla d|^2|\nabla^2d|^2+|\nabla d|^2|\nabla u|^2\nonumber\\
&+\rho|u|^2|\nabla u|^2+|\nabla u||\nabla^2d|^2+|\nabla d|^6+|\nabla^2d|^2)dxds\nonumber\\
=&C(\|\nabla u_0\|_2^2+\|\nabla^2d_0\|_2^2)+\sum_{i=1}^6I_i+C\int_0^t\|\nabla^2d\|_2^2ds.\label{L.6}
\end{align}
We estimates the terms on the right hand side of the above inequality as follows: If $N=3$, then it follows from Soblev embedding inequality and Lemma \ref{lemL.0} that
\allowdisplaybreaks\begin{align*}
I_1\leq&C\sup_{0\leq s\leq t}\|\nabla u\|_2\|\nabla d\|_2^{1/2}\|\nabla d\|_6^{3/2}\leq CC_0^{1/4}\sup_{0\leq s\leq t}\|\nabla u\|_2\|\nabla^2d\|_2^{3/2}\\
\leq&\varepsilon \sup_{0\leq s\leq t}\|\nabla^2d\|_2^2+CC_0\sup_{0\leq s\leq t}\|\nabla u\|_2^4,\\
I_2\leq&C\int_0^t\|\nabla d\|_2\|\nabla d\|_6\|\nabla^2d\|_6^2ds\leq CC_0^{1/2}\sup_{0\leq s\leq t}\|\nabla^2d\|_2\int_0^t\|\nabla^3d\|_2^2ds,\\
I_3\leq&C\int_0^t\|\nabla d\|_2\|\nabla d\|_6\|\nabla u\|_6^2ds\leq CC_0^{1/2}\sup_{0\leq s\leq  t}\|\nabla^2d\|_2\int_0^t\|\nabla^2u\|_2^2ds,\\
I_4\leq&C\int_0^t\|\sqrt\rho u\|_2\|u\|_6\|\nabla u\|_6^2ds\leq CC_0^{1/2}\sup_{0\leq s\leq t}\|\nabla u\|_2\int_0^t\|\nabla^2u\|_2^2ds,\\
I_5\leq&C\int_0^t\|\nabla u\|_2\|\nabla^2d\|_4^2ds\leq C\int_0^t\|\nabla u\|_2\|\nabla^2d\|_2^{1/2}\|\nabla^3d\|_2^{3/2}ds\\
\leq&\varepsilon\int_0^t\|\nabla^3d\|_2^2ds+C\int_0^t\|\nabla u\|_2^4\|\nabla^2d\|_2^2ds\\
\leq&\varepsilon\int_0^t\|\nabla^3d\|_2^2ds+CC_0\sup_{0\leq s\leq t}\|\nabla u\|_2^2\|\nabla^2d\|_2^2,\\
I_6\leq&C\int_0^t\|\nabla^2d\|_2^6ds\leq C\int_0^T\|\nabla^2d\|_2^4\|\nabla d\|_2^2\|\nabla^3d\|_2^2ds\\
\leq& CC_0\sup_{0\leq s\leq t}\|\nabla^2d\|_2^2\int_0^t\|\nabla^3d\|_2^2ds,
\end{align*}
and if $N=2$, then
\begin{align*}
I_1\leq&C\sup_{0\leq s\leq t}\|\nabla u\|_2\|\nabla d\|_4^2\leq C\sup_{0\leq s\leq t}\|\nabla u\|_2\|\nabla d\|_2\|\nabla^2d\|_2\\
\leq& CC_0^{1/2}\sup_{0\leq s\leq t}\|\nabla u\|_2\|\nabla^2d\|_2,\\
I_2\leq&C\int_0^t\|\nabla d\|_4^2\|\nabla^2d\|_4^2ds\leq C\int_0^t\|\nabla d\|_2\|\nabla^2d\|_2^2\|\nabla^3d\|_2ds\\
\leq& C\int_0^t\|\nabla d\|_2^2\|\nabla^3d\|_2^2ds\leq CC_0\int_0^T\|\nabla^3d\|_2^2ds,\\
I_3\leq&C\int_0^t\|\nabla d\|_4^2\|\nabla u\|_4^2ds\leq C\int_0^t\|\nabla d\|_2\|\nabla^2d\|_2\|\nabla u\|_2\|\nabla^2u\|_2ds\\
\leq&\varepsilon\int_0^t\|\nabla^2u\|_2^2ds+C\int_0^t\|\nabla d\|_2^2\|\nabla^2d\|_2^2\|\nabla u\|_2^2ds\\
\leq&\varepsilon\int_0^t\|\nabla^2u\|_2^2ds+CC_0^2\sup_{0\leq s\leq t}\|\nabla^2d\|_2^2,\\
I_4\leq&C\int_0^t\|\sqrt\rho u\|_2\|u\|_\infty\|\nabla u\|_4^2ds\leq CC_0^{1/2}\int_0^t\|u\|_\infty\|\nabla u\|_2\||\nabla^2u\|_2ds\\
\leq&\varepsilon\int_0^t\|\nabla^2u\|_2^2ds+CC_0\int_0^t\|u\|_\infty^2\|\nabla^2u\|_2^2ds,\\
I_5\leq&C\int_0^t\|\nabla u\|_2\|\nabla^2d\|_4^2ds\leq C\int_0^t\|\nabla u\|_2\|\nabla^2d\|_2\|\nabla^3d\|_2ds\\
\leq&\varepsilon\int_0^t\|\nabla^3d\|_2^2ds+C\int_0^t\|\nabla u\|_2^2\|\nabla^2d\|_2^2ds\\
\leq&\varepsilon\int_0^t\|\nabla^3d\|_2^2ds+CC_0\sup_{0\leq s\leq t}\|\nabla^2d\|_2^2,\\
I_6\leq&C\int_0^t\|\nabla d\|_2^2\|\nabla^2d\|_2^4ds\leq C\int_0^t\|\nabla d\|_2^4\|\nabla^3d\|_2^2ds\leq CC_0^2\int_0^t\|\nabla^3d\|_2^2ds.
\end{align*}
Substituting the above inequalities into (\ref{L.6}) and setting
\begin{align*}
E_1(t)=\sup_{0\leq s\leq t}(\|\nabla u\|_2^2+\|\nabla^2d\|_2^2)+\int_0^t(\|\sqrt\rho u_t\|_2^2+\|\nabla^2u\|_2^2+\|\nabla p\|_2^2+\|\nabla^3d\|_2^2)ds,
\end{align*}
we obtain that, if $N=3$ then
\begin{align*}
E_1(t)\leq& C(\|\nabla u_0\|_2^2+\|\nabla^2d_0\|_2^2)+CC_0E_1(t)^2+CC_0^{1/2}E_1(t)^{3/2}+C\int_0^t\|\nabla^2d\|_2^2ds\\
\leq&C(\|\nabla u_0\|_2^2+\|\nabla^2d_0\|_2^2)+CC_0E_1(t)^2+\frac{1}{2}E_1(t)+C\int_0^t\|\nabla^2d\|_2^2ds,
\end{align*}
and thus
\begin{equation*}
E_1(t) \leq C(\|\nabla u_0\|_2^2+\|\nabla^2d_0\|_2^2)+CC_0E_1(t)^2+C\int_0^t\|\nabla^2d\|_2^2ds,
\end{equation*}
and if $N=2$, then
\begin{equation*}
E_1(t)\leq C(\|\nabla u_0\|_2^2+\|\nabla^2d_0\|_2^2)+CC_0E_1(t)+CC_0^{1/2}\int_0^t\|u\|_\infty^2\|\nabla u\|_2^2dt+C\int_0^t\|\nabla^2d\|_2^2ds.
\end{equation*}
The proof is complete.
\end{proof}

Before continuing the energy estimates, we cite the following Sobolev inequality of logarithmic type, which will be used in Lemma \ref{lemE1}.

\begin{lemma}\label{lemL.2}(See \cite{HUANGWANG}) Assume $\Omega$ is a bounded smooth domain in $\mathbb R^2$ and $f\in L^2(s,t; H^1(\Omega))\cap L^2(0, T; W^{1,q}(\Omega))$, with some $q>2$ and $0\leq s<t\leq\infty$. Then it holds that
$$
\|f\|_{L^2(s, t; L^\infty(\Omega))}\leq C(1+\|f\|_{L^2(s,t; H^1(\Omega))}(\ln^+\|f\|_{L^2(s, t; W^{1,q}(\Omega))})^{1/2}),
$$
with some constant $C$ depending only on $q$ and $\Omega$, and independent of $s, t$.
\end{lemma}

Now, we state and prove the following lemma.

\begin{lemma}\label{lemE1}
Let $(\rho,u,p,d)$ be a
strong solution to (\ref{1.1})--(\ref{1.6}) in $Q_T$. Let $E_1(t)$ be the function defined in Lemma \ref{lemL.1}. Then there is a positive constant $\varepsilon_0$ depending only on $\overline\rho$ and $\Omega$, such that

(i) If $N=2$, then
$$
E_1(t)\leq C[1+C_0(\|\nabla u_0\|_2^2+\|\nabla^2d_0\|_2^2)^2],
$$
provided $\|\sqrt\rho_0u_0\|_2^2+\|\nabla d_0\|_2^2\leq\varepsilon_0$.

(ii) If $N=3$, then
$$
E_1(t)\leq 2C(\|\nabla u_0\|_2^2+\|\nabla^2d_0\|_2^2),
$$
provided $(\|\sqrt\rho_0u_0\|_2^2+\|\nabla d_0\|_2^2)(\|\nabla u_0\|_2^2+\|\nabla^2d_0\|_2^2)\leq\varepsilon_0.$
\end{lemma}

\begin{proof}
We first consider (i). It follows from Lemma \ref{lemL.0} and Ladyzhenskaya inequality that
\begin{align*}
\int_0^t\|\nabla^2d\|_2^2ds\leq&CC_0+C\int_0^t\|\nabla d\|_4^4ds\leq CC_0+C\sup_{0\leq s\leq t}\|\nabla d\|_2^2\int_0^t\|\nabla^2d\|_2^2ds\\
\leq&CC_0+CC_0\int_0^t\|\nabla^2d\|_2^2ds,
\end{align*}
and thus $\int_0^t\|\nabla^2d\|_2^2ds\leq CC_0$, provided $\varepsilon_0$ is small. On account of this fact,
using Poincar\'e inequality, it follows from Lemma \ref{lemL.1} that
\begin{equation*}
E_1(t)\leq C(\|\nabla u_0\|_2^2+\|\nabla^2d_0\|_2^2)+CC_0^{1/2}\int_0^t\|u\|_\infty^2\|\nabla u\|_2^2dt.
\end{equation*}
By Lemma \ref{lemL.2}, it follows from Poincar\'e inequality and Lemma \ref{lemL.0} that
\begin{align*}
\int_0^t\|u\|_\infty^2ds\leq& C\left(1+\int_0^t\|\nabla u\|_2^2dt\right)\ln[1+\int_0^t(\|\nabla u\|_2^2+\|\nabla^2u\|_2^2)ds]\\
\leq&C\ln(1+\int_0^t\|\nabla^2u\|_2^2ds).
\end{align*}
On account of this inequality, it follows from Gronwall inequality that
\begin{align*}
E_1(t)\leq&C(\|\nabla u_0\|_2^2+\|\nabla^2d_0\|_2^2)C_0^{1/2}\int_0^t\|u\|_\infty^2ds e^{CC_0^{1/2}\int_0^t\|u\|_\infty^2ds}\\
\leq&C(\|\nabla u_0\|_2^2+\|\nabla^2d_0\|_2^2)C_0^{1/2}\ln(1+\int_0^t\|\nabla^2u\|_2^2ds)(1+\int_0^t\|\nabla^2u\|_2^2ds)^{CC_0^{1/2}}\\
\leq&CC_0^{1/2}(\|\nabla u_0\|_2^2+\|\nabla^2d_0\|_2^2)(1+\int_0^t\|\nabla^2u\|_2^2ds)^{1/2}\\
\leq&\frac{1}{2}E_1(t)+CC_0^{1/2}(\|\nabla u_0\|_2^2+\|\nabla^2d_0\|_2^2)+CC_0(\|\nabla u_0\|_2^2+\|\nabla^2d_0\|_2^2)^2,
\end{align*}
provided $\varepsilon_0$ is small enough, and thus
$$
E_1(t)\leq C[1+C_0(\|\nabla u_0\|_2^2+\|\nabla^2d_0\|_2^2)^2],\quad t\in(0, T).
$$

Now, we prove (ii). Using Gagliado-Nirenberg, it follows from Lemma \ref{lemL.0} that
\begin{align*}
\int_0^t\|\nabla^2d\|_2^2ds\leq&CC_0+C\int_0^t\|\nabla d\|_4^4ds\leq CC_0+C\int_0^t\|\nabla d\|_2\|\nabla^2d\|_2^3ds\\
\leq&CC_0+C\int_0^t\|\nabla d\|_2\|\nabla^2d\|_2\|\nabla d\|_2\|\nabla^3d\|_2ds\\
\leq&CC_0+\frac{1}{2}\int_0^t\|\nabla^2d\|_2^2ds+C\int_0^t\|\nabla d\|_2^4\|\nabla^3d\|_2^2ds,
\end{align*}
and thus
\begin{align*}
\int_0^t\|\nabla^2d\|_2^2ds\leq&CC_0+CC_0\sup_{0\leq s\leq t}\|\nabla d\|_2^2\int_0^t\|\nabla^3d\|_2^2ds\\
\leq&C(\|\nabla u_0\|_2^2+\|\nabla^2d_0\|_2^2)+CC_0\sup_{0\leq s\leq t}\|\nabla^2d\|_2^2\int_0^t\|\nabla^3d\|_2^2ds\\
\leq&C(\|\nabla u_0\|_2^2+\|\nabla^2d_0\|_2^2)+CC_0E_1(t)^2,
\end{align*}
in the above we have used the fact that $\|u\|_2^2+\|\nabla d\|_2^2\leq C(\|\nabla u\|_2^2+\|\nabla^2d\|_2^2)$ guaranteed by Poincar\'e inequality and the boundary condition (\ref{1.6}). On account of the above inequality, by Lemma \ref{lemL.1}, there is a positive constant $C_*\geq 1$ depending only on $\bar\rho$ and $\Omega$, such that
$$
E_1(t)\leq C_*(\|\nabla u_0\|_2^2+\|\nabla^2d_0\|_2^2)+C_*C_0E_1(t)^2,
$$
which implies
$$
E_1(t)\leq\frac{1-\sqrt{1-4C_*^2C_0(\|\nabla u_0\|_2^2+\|\nabla^2d_0\|_2^2)}}{2C_*C_0},
$$
or
$$
E_1(t)\geq\frac{1+\sqrt{1-4C_*^2C_0(\|\nabla u_0\|_2^2+\|\nabla^2d_0\|_2^2)}}{2C_*C_0}
$$
Note that $E_1(t)$ is a nondecreasing and continuous function on $[0, T]$. Set $\varepsilon_0=\frac{1}{8C_*^2}$, then it has
$$
4C_*^2C_0(\|\nabla u_0\|_2^2+\|\nabla^2d_0\|_2^2)\leq 1/2.
$$
One can easily check that
$$
E_1(0)=\|\nabla u_0\|_2^2+\|\nabla^2d_0\|_2^2<\frac{1-\sqrt{1-4C_*^2C_0(\|\nabla u_0\|_2^2+\|\nabla^2d_0\|_2^2)}}{2C_*C_0}.
$$
Consequently, the continuity of $E_1(t)$ implies that
\begin{align*}
E_1(t)\leq&\frac{1-\sqrt{1-4C_*^2C_0(\|\nabla u_0\|_2^2+\|\nabla^2d_0\|_2^2)}}{2C_*C_0}\\
\leq&\frac{1-(1-4C_*^2C_0(\|\nabla u_0\|_2^2+\|\nabla^2d_0\|_2^2))}{2C_*C_0}=2C_*(\|\nabla u_0\|_2^2+\|\nabla^2d_0\|_2^2),
\end{align*}
the proof is complete.
\end{proof}

\begin{lemma} \label{lemE2}Let $(\rho,u,p,d)$ be a
strong solution to (\ref{1.1})--(\ref{1.6}) in $Q_T$. Set
\begin{align*}
E_2(t)=&\sup_{0\leq s\leq t}(\|\sqrt\rho u_t\|_2^2+\|\nabla^2u\|_2^2+\|\nabla p\|_2^2+\|\nabla^3d\|_2^2)\\
&+\int_0^t(\|\nabla u_t\|_2^2+\|d_{tt}\|_2^2+\|\nabla^2d_t\|_2^2)ds.
\end{align*}
Then there is a positive constant $\varepsilon_0$ depending only on $\overline\rho$ and $\Omega$, such that the following hold

(i) If $N=2$, then
$$
E_2(t)\leq C(1+\|g_0\|_2^2+\|\nabla^2u_0\|_2^2+\|\nabla^2d_0\|_{H^1}^2)^8e^{(1+\|\nabla u_0\|_2^2+\|\nabla^2d_0\|_2^2)^2}
$$
for all $t\in(0, T)$, provided $\|\sqrt{\rho_0}u_0\|_2^2+\|\nabla^2d_0\|_2^2\leq\varepsilon_0$.

(ii) If $N=3$, then
$$
E_2(t)\leq C[\|g_0\|_2^2+(1+\|\nabla u_0\|_2^2+\|\nabla d_0\|_{H^2}^2)^4]
$$
for all $t\in(0, T)$, provided $(\|\sqrt{\rho_0}u_0\|_2^2+\|\nabla^2d_0\|_2^2)(\|\nabla u_0\|_2^2+\|\nabla^2d_0\|_2^2)\leq\varepsilon_0$.
\end{lemma}

\begin{proof}
Differentiating the momentum equation (\ref{1.2}) with respect to
$t$, multiplying the resulting equation by $u_t$ and integrating over $\Omega$ yields
\begin{align*}
&\frac{1}{2}\frac{d}{dt}\int_\Omega\rho|u_t|^2dx+\int_\Omega|\nabla u_t|^2dx\\
=&\int_\Omega(\nabla d_t\odot\nabla d+\nabla d\odot\nabla d_t):\nabla u_tdx-\int_\Omega\rho (u_t\cdot\nabla) u\cdot u_tdx\nonumber\\
&+\int_\Omega\textmd{div}(\rho u)(u_t+(u\cdot\nabla) u)\cdot u_tdx.
\end{align*}
It follows
\begin{align*}
&\int_\Omega\textmd{div}(\rho u)(u_t+(u\cdot\nabla) u)\cdot u_tdx=-\int_\Omega\Big\{\rho u\cdot\nabla(|u_t|^2)+\rho u\cdot\nabla[(u\cdot\nabla u)\cdot u_t] \Big\}dx\\
\leq&C\int_\Omega (\rho|u||\nabla u_t||u_t|+\rho|u||\nabla u|^2|u_t|+\rho|u|^2|\nabla^2u||u_t|+\rho|u|^2|\nabla u||\nabla u_t|)dx.
\end{align*}
Combining the above two inequalities yields
\begin{align}
&\frac{d}{dt}\int_\Omega\rho|u_t|^2dx+\int_\Omega|\nabla u_t|^2dx\nonumber\\
\leq&C\int_\Omega(|\nabla d|^2|\nabla d_t|^2+\rho|u|^2|u_t|^2+\rho|u||\nabla u|^2|u_t|\nonumber\\
&+\rho|u|^2|\nabla^2u||u_t|+\rho|u|^4|\nabla u|^2+\rho |u_t|^2|\nabla u|)dx.\label{M.1}
\end{align}

Sine the estimates on $d$ are different for the case $N=2$ and the case $N=3$, we prove (i) and (ii) separately.

(i) The case that $N=2$. Taking $\Delta$ on both sides of equation (\ref{1.4}), there holds
\begin{align*}
\Delta d_t-\Delta^2d=&|\nabla d|^2\Delta d+2(\nabla d:\nabla\Delta d) d+2|\nabla^2d|^2d\\
&-(u\cdot\nabla)\Delta d-2(\nabla u_i\cdot\partial_i\nabla)d-(\Delta u\cdot\nabla)d.
\end{align*}
Squaring both sides of the above equation and integrating over $\Omega$ yields
\begin{align}
&-2\int_\Omega\Delta d_t\Delta^2ddx+\int_\Omega(|\Delta d_t|^2+|\Delta^2d|^2)dx\nonumber\\
\leq&C\int_\Omega(|u|^2|\nabla\Delta d|^2+|\nabla u|^2|\nabla^2d|^2+|\nabla d|^2|\Delta u|^2\nonumber\\
&+|\nabla d|^4|\Delta d|^2+|\nabla d|^2|\nabla\Delta d|^2+|\nabla^2d|^4)dx.\label{N.1}
\end{align}

By equation (\ref{1.4}) and the boundary value condition (\ref{1.6}), there holds
$$
\Delta d|_{\partial\Omega}=-|\nabla d|^2d|_{\partial\Omega}=-d_0^*|\nabla d|^2|_{\partial\Omega}.
$$
By the aid of this boundary value condition, it follow from elliptic estimates and Ladyzhenskaya inequality that
\begin{align*}
\int_\Omega|\nabla^4d|^2dx=& \int_\Omega\|\nabla^4 (d-d_0^*)\|_{2}^2dx\leq C\|\Delta(d-d_*)\|_{H^2}^2\\
=&C\|\Delta d\|_{H^2}^2\leq C(\|\Delta^2d\|_2^2+\||\nabla d|^2\|_{H^2}^2)\leq C(\|\Delta^2d\|_2^2+\||\nabla d|^2\|_2^2+\|\nabla^2|\nabla d|^2\|_{2}^2)\\
\leq&C(\|\Delta^2d\|_2^2+\|\nabla d\|_4^4+\||\nabla d|\nabla^3d\|_2^2+\||\nabla^2d|^2\|_2^2)\\
\leq& C(\|\Delta^2d\|_2^2+\|\nabla d\|_4^4+\|\nabla d\|_4^2\|\nabla^3d\|_4^2+\|\nabla^2d\|_4^4)\\
\leq&C(\|\Delta^2d\|_2^2+\|\nabla d\|_2^2\|\nabla^2 d\|_2^2+\|\nabla d\|_2\|\nabla^2 d\|_2\|\nabla^3 d\|_2\|\nabla^4d\|_2+\|\nabla^2d\|_2^2\|\nabla^3 d\|_2^2)\\
\leq&\frac{1}{2}\|\nabla^4d\|_2^2+C(\|\Delta^2d\|_2^2+\|\nabla d\|_2^2\|\nabla^2 d\|_2^2\\
&+\|\nabla d\|_2^2\|\nabla^2 d\|_2^2\|\nabla^3 d\|_2^2+\|\nabla^2d\|_2^2\|\nabla^3 d\|_2^2),
\end{align*}
and thus it follows from Lemma \ref{lemE1} that
\begin{equation}
\|\nabla^4d\|_2^2\leq C\|\Delta^2d\|_2^2+C(1+\|\nabla u_0\|_2^2+\|\nabla^2d_0\|_2^2)^2(1+\|\nabla^3d\|_2^2).\label{N.2}
\end{equation}
Note that the following boundary condition holds true
$$
\Delta d_t|_{\partial\Omega}=-(|\nabla d|^2d)_t|_{\partial\Omega}.
$$
Integration by parts, it follows
\begin{align}
-2\int_\Omega\Delta d_t\Delta^2ddx=&-2\int_{\partial\Omega}\frac{\partial}{\partial n}\Delta d\cdot\Delta d_tdS+2\int_\Omega\nabla\Delta d_t\nabla\Delta ddx\nonumber\\
=&\frac{d}{dt}\int_\Omega|\nabla\Delta d|^2dx+2\int_{\partial\Omega}\frac{\partial}{\partial n}\Delta d(|\nabla d|^2d)_tdS.\label{N.2-1}
\end{align}
Recalling the trace inequality
$$
\|f\|_{L^1(\partial\Omega)}\leq C\|f\|_{W^{1,1}(\Omega)},
$$
there holds
\allowdisplaybreaks\begin{align*}
&\int_{\partial\Omega}\frac{\partial}{\partial n}\Delta d(|\nabla d|^2d)_tdS=\int_{\partial\Omega}\frac{\partial}{\partial n}\Delta d(|\nabla d|^2d_0^*)_tdS\\
=&2\int_{\partial\Omega}\frac{\partial}{\partial n}\Delta d\nabla d:\nabla d_td_0^*dS=2\int_{\partial\Omega}\frac{\partial}{\partial n}\Delta d\cdot d\nabla d:\nabla d_t dS\\
=&\int_{\partial\Omega}\left(\frac{\partial}{\partial n}(\Delta d\cdot d)-\Delta d\frac{\partial}{\partial n}d\right)2\nabla d:\nabla d_tdS\\
=&-\int_{\partial\Omega}\left(\frac{\partial}{\partial n}|\nabla d|^2+\Delta d\frac{\partial}{\partial n}d\right)2\nabla d:\nabla d_tdS\\
\geq&-C\int_{\partial\Omega}|\nabla^2d||\nabla d|^2|\nabla d_t|dS\geq-C\||\nabla^2d||\nabla d|^2|\nabla d_t|\|_{W^{1,1}(\Omega)}\\
\geq&-C\int_\Omega(|\nabla^2d||\nabla d|^2|\nabla d_t|+|\nabla^3d||\nabla d|^2|\nabla d_t|\\
&+|\nabla^2d|^2|\nabla d||\nabla d_t|+|\nabla^2d||\nabla d|^2|\nabla^2d_t|)dx\\
\geq&-\varepsilon\|\nabla^2d_t\|_2^2-C\int_\Omega(|\nabla d|^2|\nabla^2d||\nabla d_t|+|\nabla d|^2|\nabla^3d||\nabla d_t|
\\&+|\nabla d||\nabla^2d|^2|\nabla d_t|+|\nabla d|^4|\nabla^2d|^2)dx.
\end{align*}

Combining the above inequality with (\ref{N.1})--(\ref{N.2-1}) and using Lemma \ref{lemE1}, we obtain
\begin{align}
&\sup_{0\leq s\leq t}\|\nabla^3d\|_2^2+(\|\nabla^2d_t\|_2^2+\|\nabla^4d\|_2^2)\nonumber\\
\leq&C[\|\nabla\Delta d_0\|_2^2+1+(\|\nabla^2u_0\|_2^2+\|\nabla^2d_0\|_2^2)^4]+C\int_0^t[|u|^2|\nabla^3d|^2\nonumber\\
&+|\nabla d|^2(|\nabla^2u|^2+|\nabla^3d|^2+|\nabla^2d||\nabla d_t|+|\nabla^3d||\nabla d_t|)\nonumber\\
&+|\nabla u|^2|\nabla^2d|^2+|\nabla^2d|^4+|\nabla d|^4|\nabla^2d|^2+|\nabla d||\nabla^2d||\nabla d_t|]dxds.\label{N.4}
\end{align}
By Lemma \ref{lemL.0}, it follows from Ladyzhenskaya inequality that
\begin{align*}
&\int_0^t\|\nabla^2d\|_2^2ds\leq C\int_0^t\|\Delta d\|_2^2ds\leq CC_0+C\int_0^t\|\nabla d\|_4^4ds\\
\leq& CC_0+C\sup_{0\leq s\leq t}\|\nabla d\|_2^2\int_0^t\|\nabla^2d\|_2^2ds\leq CC_0+CC_0\int_0^t\|\nabla^2d\|_2^2ds,
\end{align*}
and thus
$$
\int_0^t\|\nabla^2d\|_2^2ds\leq CC_0,
$$
provided $\varepsilon_0$ is small enough.
On account of this inequality, using Ladyzhenskaya and Gagliado-Nirenberg inequality, by Lemma \ref{lemL.0}, we estimate the terms on the right hand side of inequality (\ref{N.4}) as follows
\begin{align*}
I_1=&\int_0^t\int_\Omega|u|^2|\nabla^3d|^2dxds\leq\int_0^t\|u\|_4^2\|\nabla^3d\|_4^2ds\\
\leq& C\int_0^t\|\nabla u\|_2^2\|\nabla^3d\|_2\|\nabla^4d\|_2ds\leq \varepsilon\int_0^t\|\nabla^4d\|_2^2ds+CE_1(t)^3,\\
I_2=&\int_0^t\int_\Omega|\nabla d|^2(|\nabla^3d|^2+|\nabla^2d||\nabla d_t|+|\nabla^3d||\nabla d_t|)dxds\\
\leq&C\int_0^t\|\nabla d\|_4^2(\|\nabla^3d\|_4^2+\|\nabla d_t\|_4^2+\|\nabla^2d\|_4^2)ds\\
\leq&C\int_0^t\|\nabla d\|_2\|\nabla^2d\|_2(\|\nabla^3d\|_2\|\nabla^4d\|_2+\|\nabla d_t\|_2\|\nabla^2d_t\|_2+\|\nabla^2d\|_2\|\nabla^3d\|_2)ds\\
\leq&\varepsilon\int_0^t(\|\nabla^4d\|_2^2+\|\nabla^2d_t\|_2^2)ds+C\int_0^t\|\nabla^2d\|_2^2(\|\nabla^3d\|_2^2+\|\nabla d_t\|_2^2+\|\nabla^2d\|_2^2)ds\\
\leq&\varepsilon\int_0^t(\|\nabla^4d\|_2^2+\|\nabla^2d_t\|_2^2)ds+C(E_1(t)+E_1(t)^2)+E_1(t)\int_0^t\|\nabla d_t\|_2^2ds,\\
I_3=&\int_0^t\int_\Omega|\nabla d|^2|\nabla^2u|^2dxds\leq C\int_0^t\|\nabla d\|_\infty^2\|\nabla^2u\|_2^2ds\\
\leq& C\int_0^t\|\nabla d\|_2\|\|\nabla^3d\|_2\|\nabla^2u\|_2^2ds\leq\varepsilon\sup_{0\leq s\leq t}\|\nabla^3d\|_2^2+C\left(\int_0^t\|\nabla^2u\|_2^2ds\right)^2\\
\leq&\varepsilon\sup_{0\leq s\leq t}\|\nabla^3d\|_2^2+CE_1(t)^2,\\
I_4=&\int_0^t\int_\Omega(|\nabla u|^2|\nabla^2d|^2+|\nabla^2d|^4)dxds\leq\int_0^t(\|\nabla u\|_4^2\|\nabla^2d\|_4^2+\|\nabla^2 d\|_4^4)dx\\
\leq&C\int_0^t(\|\nabla u\|_2\|\nabla^2d\|_2\|\nabla^2u\|_2\|\nabla^3d\|_2+\|\nabla^2d\|_2^2\|\nabla^3d\|_2^2)ds\leq CE_1(t)^2,\\
I_5=&\int_0^t\int_\Omega|\nabla d|^4|\nabla^2d|^2dxds\leq \int_0^t\|\nabla d\|_8^4\|\nabla^2d\|_4^2ds\\
\leq&\int_0^t\|\nabla d\|_2\|\nabla^2d\|_2^3\|\nabla^2d\|_2\|\nabla^3d\|_2^2ds\leq CE_1(t)^3,\\
I_6=&\int_0^t\int_\Omega|\nabla d||\nabla^2d||\nabla d_t|dxds\leq \int_0^t\|\nabla d\|_4\|\nabla^2d\|_4\|\nabla d_t\|_2ds\\
\leq&C\int_0^t\|\nabla d\|_2^{1/2}\|\nabla^2d\|_2^{1/2}\|\nabla^2d\|_2^{1/2}\|\nabla^3d\|_2^{1/2}\|\nabla d_t\|_2ds\\
\leq&C\int_0^t(\|\nabla^2d\|_2^4+\|\nabla^3d\|_2^2+\|\nabla d_t\|_2^2)ds\leq CE_1(t)+C\int_0^t\|\nabla d_t\|_2^2ds.
\end{align*}
Substituting all the above inequalities into (\ref{N.4}) yields
\begin{align}
&\sup_{0\leq s\leq t}\|\nabla^3d\|_2^2+(\|\nabla^2d_t\|_2^2+\|\nabla^4d\|_2^2)\nonumber\\
\leq&C[\|\nabla\Delta d_0\|_2^2+1+(\|\nabla^2u_0\|_2^2+\|\nabla^2d_0\|_2^2)^4]+CE_1(t))^3\nonumber\\
&+C(1+E_1(t))\int_0^t\|\nabla d_t\|_2^2ds.\label{N.5}
\end{align}
By equation (\ref{1.4}), using Ladyzhenskaya inequality and Sobolev inequality, recalling that
$$
\int_0^t\|\nabla^2d\|_2^2ds\leq CC_0,
$$
it follows from Lemma \ref{lemE1} and Lemma \ref{lemL.0} that
\begin{align}
\int_0^t\|\nabla d_t\|_2^2ds\leq&C\int_0^t(\|\nabla^3d\|_2^2+\|\nabla(u\cdot\nabla d-|\nabla d|^2d)\|_2^2)ds\nonumber\\
\leq&C\int_0^t(\|\nabla^3d\|_2^2+\|u\|_4^2\|\nabla^2d\|_4^2+\|\nabla u\|_4^2\|\nabla d\|_4^2+\|\nabla d\|_6^6+\|\nabla d\|_4^2\|\nabla^2d\|_4^2)ds\nonumber\\
\leq&C\int_0^t(\|\nabla^3d\|_2^2+\|\nabla u\|_2^2\|\nabla^2d\|_2\|\nabla^3d\|_2+\|\nabla u\|_2\|\nabla^2u\|_2\|\nabla d\|_2\|\nabla^2d\|_2\nonumber\\
&+\|\nabla d\|_2^2\|\nabla^2d\|_2^4+\|\nabla d\|_2\|\nabla^2d\|_2^2\|\nabla^3d\|_2)ds\nonumber\\
\leq&C\int_0^t(\|\nabla^3d\|_2^2+\|\nabla^2u\|_2^2+\|\nabla u\|_2^4\|\nabla^2d\|_2^2\nonumber\\
&+\|\nabla u\|_2^2\|\nabla^2d\|_2^4+\|\nabla^2d\|_2^2+\|\nabla^2d\|_2^4)ds\leq C(1+E_1(t)^2),\label{N.6}
\end{align}
from which, recalling (\ref{N.5}) and using Lemma \ref{lemE1}, we arrive
\begin{align}
&\sup_{0\leq s\leq t}\|\nabla^3d\|_2^2+\int_0^t(\|\nabla^2d_t\|_2^2+\|\nabla^4d\|_2^2)ds\nonumber\\
\leq&C[\|\nabla\Delta d_0\|_2^2+(1+\|\nabla^2u_0\|_2^2+\|\nabla^2d_0\|_2^2)^6].\label{N.7}
\end{align}

By (\ref{M.1}), it has
\begin{align}
&\sup_{0\leq s\leq t}\|\sqrt\rho u_t\|_2^2+\int_0^t\|\nabla u_t\|_2^2\nonumber\\
\leq&\|\sqrt{\rho_0}u_t(0)\|_2^2+\int_0^t\int_\Omega(|\nabla d|^2|\nabla d_t|^2+\rho|u|^2|u_t|^2+\rho|u|^2|\nabla^2u||u_t|\nonumber\\
&+\rho |\nabla u||u_t|^2
+\rho|u||\nabla u|^2|u_t|+\rho|u|^4|\nabla u|^2)dxds.\label{N.8}
\end{align}
By Ladyzhenskaya inequality, Sobolev embedding inequality and recalling (\ref{N.6}), we can estimate the terms on the right hand side of the above inequality as follows
\begin{align*}
I_1\leq&\int_0^t\int_\Omega(|\nabla d|^2|\nabla d_t|^2+\rho|u|^2|u_t|^2+\rho|u|^2|\nabla^2u||u_t|)dxds\\
\leq&C\int_0^t(\|\nabla d\|_4^2\|\nabla d_t\|_4^2+\|u\|_8^2\|\sqrt\rho u_t\|_2\|u_t\|_4+\|u\|_8\|\nabla^2u\|_2\|u_t\|_4)ds\\
\leq&C\int_0^t(\|\nabla d\|_2\|\nabla^2d\|_2\|\nabla d_t\|_2\|\nabla^2d_t\|_2+\|\nabla u\|_2^2\|\sqrt\rho u_t\|_2\|\nabla u_t\|_2\\
&+\|\nabla u\|_2\|\nabla^2u\|_2\|\nabla u_t\|_2)ds\\
\leq&\varepsilon\int_0^t(\|\nabla^2d_t\|_2^2+\|\nabla u_t\|_2^2)ds+C\int_0^t(\|\nabla d\|_2^2\|\nabla^2d\|_2^2\|\nabla d_t\|_2^2\\
&+\|\nabla u\|_2^4\|\sqrt\rho u_t\|_2^2+\|\nabla u\|_2^2\|\nabla^2u\|_2^2)ds\\
\leq&\varepsilon\int_0^t(\|\nabla^2d_t\|_2^2+\|\nabla u_t\|_2^2)ds+C(1+E_1(t)^3),\\
I_2=&\int_0^t\int_\Omega\rho|\nabla u||u_t|^2dxds\leq C\int_0^t\|\nabla u\|_4\|\sqrt\rho u_t\|_2\|u_t\|_4ds\\
\leq&C\int_0^t\|\nabla u\|_2^{1/2}\|\nabla^2u\|_2^{1/2}\|\sqrt\rho u_t\|_2\|\nabla u_t\|_2ds\\
\leq&\varepsilon\int_0^t\|\nabla u_t\|_2^2ds+C\int_0^t\|\nabla u\|_2\|\nabla^2u\|_2\|\sqrt\rho u_t\|_2^2ds\\
\leq&\varepsilon\int_0^t\|\nabla u_t\|_2^2ds+C\int_0^t(\|\nabla u\|_2^2+\|\nabla^2u\|_2^2)\|\sqrt\rho u_t\|_2^2ds\\
\leq&\varepsilon\int_0^t\|\nabla u_t\|_2^2ds+C\int_0^t\|\nabla^2u\|_2^2\|\sqrt\rho u_t\|_2^2ds+CE_1(t)^2,\\
I_3=&\int_0^t\int_\Omega\rho|u||\nabla u|^2|u_t|ds\leq C\int_0^t\|u\|_4\|\nabla u\|_4^2\|u_t\|_4ds\\
\leq&C\int_0^t\|\nabla u\|_2^2\|\nabla^2u\|_2\|\nabla u_t\|_2ds\leq\varepsilon\int_0^t\|\nabla u_t\|_2^2ds+ C\int_0^t\|\nabla u\|_2^4\|\nabla^2u\|_2^2ds\\
\leq&\varepsilon\int_0^t\|\nabla u_t\|_2^2ds+C E_1(t)^3,\\
I_4=&\int_0^t\int_\Omega\rho|u|^4|\nabla u|^2dxds\leq C\int_0^t\|u\|_8^4\|\nabla u\|_4^2ds\\
\leq&C\int_0^t\|\nabla u\|_2^4\|\nabla u\|_2\|\nabla^2u\|_2ds\leq C\int_0^t\|\nabla u\|_2^4(\|\nabla u\|_2^2+\|\nabla^2u\|_2^2)ds\\
\leq&C((1+E_1(t)^3).
\end{align*}
Substituting these inequalities into (\ref{N.8}), using the inequality (\ref{N.7}) and recalling the compatible condition, it follows from Lemma \ref{lemE1} that
\begin{align*}
&\sup_{0\leq s\leq t}\|\sqrt\rho u_t\|_2^2+\int_0^t\|\nabla u_t\|_2^2ds\\
\leq&C[\|g_0\|_2^2+\|\nabla\Delta d_0\|_2^2+(1+\|\nabla^2u_0\|_2^2+\|\nabla^2d_0\|_2^2)^6]+C\int_0^t\|\nabla^2u\|_2^2\|\sqrt\rho u_t\|_2^2ds,
\end{align*}
from which, by Gronwall inequality, we deduce
\begin{align}
&\sup_{0\leq s\leq t}\|\sqrt\rho u_t\|_2^2+\int_0^t\|\nabla u_t\|_2^2ds\nonumber\\
\leq&C[\|\nabla\Delta d_0\|_2^2+(1+\|\nabla^2u_0\|_2^2+\|\nabla^2d_0\|_2^2)^6]e^{C\int_0^t\|\nabla^2u\|_2^2ds}\int_0^t\|\nabla^2u\|_2^2ds\nonumber\\
\leq&C(1+\|g_0\|_2^2+\|\nabla^2u_0\|_2^2+\|\nabla^2d_0\|_{H^1}^2)^6(1+\|\nabla u_0\|_2^2+\|\nabla^2d_0\|_2^2)^2e^{(1+\|\nabla u_0\|_2^2+\|\nabla^2d_0\|_2^2)^2}\nonumber\\
\leq&C(1+\|g_0\|_2^2+\|\nabla^2u_0\|_2^2+\|\nabla^2d_0\|_{H^1}^2)^8e^{(1+\|\nabla u_0\|_2^2+\|\nabla^2d_0\|_2^2)^2}.\label{N.9}
\end{align}

Elliptic estimates to Stokes equations give
\begin{align*}
\|\nabla^2u\|_2^2+\|\nabla p\|_2^2\leq&C(\|\sqrt\rho u_t\|_2^2+\|\rho u|\nabla u|\|_2^2+\|\nabla d|\nabla^2d|\|_2^2)\\
\leq&C(\|\sqrt\rho u_t\|_2^2+\|u\|_4^2\|\nabla u\|_4^2+\|\nabla d\|_4^2\|\nabla^2d\|_4^2)\\
\leq&C(\|\sqrt\rho u_t\|_2^2+\|\nabla u\|_2^2\|\nabla u\|_2\|\nabla^2u\|_2+\|\nabla d\|_2\|\nabla^2d\|_2^2\|\nabla^3d\|_2)\\
\leq&\varepsilon(\|\nabla^2u\|_2^2+\|\nabla^3d\|_2^2)+C(\|\sqrt\rho u_t\|_2^2+\|\nabla u\|_2^6+\|\nabla^2d\|_2^4),
\end{align*}
which, combined with (\ref{N.7}), together with (\ref{N.9}), we obtain
\begin{align*}
&\sup_{0\leq s\leq t}(\|\sqrt\rho u_t\|_2^2+\|\nabla^2u\|_2^2+\|\nabla p\|_2+\|\nabla^3d\|_2^2)+\int_0^t(\|\nabla u_t\|_2^2+\|\nabla^2d_t\|_2^2+\|\nabla^4d\|_2^2)ds\\
\leq&C(1+\|g_0\|_2^2+\|\nabla^2u_0\|_2^2+\|\nabla^2d_0\|_{H^1}^2)^8e^{(1+\|\nabla u_0\|_2^2+\|\nabla^2d_0\|_2^2)^2}.
\end{align*}

(ii) The case that $N=3$. Differentiate equation (\ref{1.4}) with respect to $t$, then it has
\begin{equation*}
d_{tt}-\Delta d_t=|\nabla d|^2d_t-(u\cdot\nabla) d_t+2(\nabla d:\nabla d_t) d-(u_t\cdot \nabla )d.
\end{equation*}
Square both sides of this equation and integration by parts, using Young inequality, we have
\begin{align}
&\frac{d}{dt}\int_\Omega|\nabla d_t|^2dx+\int_\Omega(|d_{tt}|^2+|\Delta d_t|^2)dx\nonumber\\
\leq&C\int_\Omega(|\nabla d|^4|d_t|^2+|u|^2|\nabla d_t|^2+|\nabla d|^2|\nabla d_t|^2+|\nabla d|^2|u_t|^2)dx.
\end{align}
Combining this inequality with (\ref{M.1}), we obtain
\begin{align}
&\sup_{0\leq s\leq t}(\|\sqrt\rho|u_t\|_2^2+\|\nabla d_t\|_2^2)+\int_0^t(\|\nabla u_t\|_2^2+\|d_{tt}\|_2^2+\|\Delta d_t\|_2^2)ds\nonumber\\
\leq&(\|\sqrt{\rho_0}u_t(0)\|_2^2+\|\nabla d_t(0)\|_2^2)+C\int_0^t\int_\Omega(\rho|u|^2|u_t|^2+\rho|u|^2|\nabla^2u||u_t|+|u|^2|\nabla d_t|^2+|\nabla d|^2|u_t|^2\nonumber\\
&+|\nabla d|^2|\nabla d_t|^2+\rho|u|^4|\nabla u|^2+|\nabla d|^4|d_t|^2+\rho|u||\nabla u|^2|u_t|+\rho |u_t|^2|\nabla u|)dxdt\nonumber\\
=&(\|\sqrt{\rho_0}u_t(0)\|_2^2+\|\nabla d_t(0)\|_2^2)+C\sum_{i=1}^9I_i.\label{M.2}
\end{align}

By Lemma \ref{lemE1}, Sobolev embedding inequality and H\"older inequality, we estimate the term on the right hand side as follows
\allowdisplaybreaks\begin{align*}
I_1\leq&C\int_0^t\|u\|_6^2\|\sqrt\rho u_t\|_2\|u_t\|_6ds\leq C\int_0^t\|\nabla u\|_2^2\|\sqrt\rho u_t\|_2\|\nabla u_t\|_2ds\\
\leq&\varepsilon\int_0^t\|\nabla u_t\|_2^2ds+C\int_0^t\|\nabla u\|_2^4\|\sqrt\rho u_t\|_2^2ds\\
\leq&\varepsilon\int_0^t\|\nabla u_t\|_2^2ds+C(\|\nabla u_0\|_2^2+\|\nabla^2d_0\|_2^2)^3,\\
I_2\leq&C\int_0^t\|u\|_6^2\|\nabla^2u\|_2\|u_t\|_6ds\leq C\int_0^t\|\nabla u\|_2^2\|\nabla^2u\|_2\|\nabla u_t\|_2ds\\
\leq&\varepsilon\int_0^t\|\nabla u_t\|_2^2ds+C(\|\nabla u_0\|_2^2+\||\nabla^2d_0\|_2^2)^3,\\
I_3\leq&C\int_0^t\|u\|_6^2\|\nabla d_t\|_2\|\nabla d_t\|_6ds\leq C\int_0^t\|\nabla u\|_2^2\|\nabla d_t\|_2\|\nabla^2d_t\|_2ds\\
\leq&\varepsilon\int_0^t\|\nabla^2d_t\|_2^2ds+C(\|\nabla u_0\|_2^2+\|\nabla^2d_0\|_2^2)^2\int_0^t\|\nabla d_t\|_2^2ds,\\
I_4\leq&C\int_0^t\|\nabla d\|_2\|\nabla d\|_6\|u_t\|_6^2ds\leq C\int_0^t\|\nabla d\|_2\|\nabla^2d\|_2\|\nabla u_t\|_2^2ds\\
\leq&CC_0^{1/2}(\|\nabla u_0\|_2^2+\|\nabla^2d_0\|_2^2)^{1/2}\int_0^t\|\nabla d_t\|_2^2ds\leq C\varepsilon_0^{1/2}\int_0^t\|\nabla u_t\|_2^2ds,\\
I_5\leq&\int_0^t\|\nabla d\|_6^2\|\nabla d_t\|_2\|\nabla d_t\|_6ds\leq C\int_0^t\|\nabla^2d\|_2^2\|\nabla d_t\|_2\|\nabla^2d_t\|_2ds\\
\leq&\varepsilon\int_0^t\|\nabla^2d_t\|_2^2ds+C\int_0^t\|\nabla^2d\|_2^4\|\nabla^2d_t\|_2^2ds\\
\leq&\varepsilon\int_0^t\|\nabla^2d_t\|_2^2ds+C(\|\nabla u_0\|_2^2+\|\nabla^2d_0\|_2^2)^2\int_0^t\|\nabla d_t\|_2^2ds,\\
I_6\leq&C\int_0^t\|u\|_6^4\|\nabla u\|_6^2ds\leq C\int_0^t\|\nabla u\|_2^4\|\nabla^2u\|_2^2ds\leq C(\|\nabla u_0\|_2^2+\|\nabla^2d_0\|_2^2)^3,\\
I_7\leq&\int_0^t\|\nabla d\|_6^4\|d_t\|_6^2ds\leq C\int_0^t\|\nabla^2d\|_2^4\|\nabla d_t\|_2^2ds\\
\leq&C(\|\nabla u_0\|_2^2+\|\nabla^2d_0\|_2^2)^2\int_0^t\|\nabla d_t\|_2^2ds,\\
I_8\leq&C\int_0^t\|\sqrt\rho u\|_2\|\nabla u\|_6^2\|u_t\|_6ds\leq C\int_0^t\|\sqrt\rho u\|_2^2\|\nabla^2u\|_2^2\|\nabla u_t\|_2ds\\
\leq&\varepsilon\int_0^t\|\nabla u_t\|_2^2ds+C\int_0^t\|\sqrt\rho u\|_2^2\|\nabla^2u\|_2^4ds\\
\leq&\varepsilon\int_0^t\|\nabla u_t\|_2^2ds+C(\|\nabla u_0\|_2^2+\|\nabla^2d_0\|_2^2)^3,\\
I_9\leq&C\int_0^t\|\sqrt\rho u_t\|_2\|u_t\|_6\|\nabla u\|_2^{1/2}\|\nabla^2u\|_2^{1/2}ds\\
\leq& C\int_0^t\|\sqrt\rho u_t\|_2\|\nabla u_t\|_2\|\nabla u\|_2^{1/2}\|\nabla^2u\|_2^{1/2}ds\\
\leq&\varepsilon\int_0^t\|\nabla u_t\|_2^2ds+C\int_0^t\|\sqrt\rho u_t\|_2^2\|\nabla u\|_2\|\nabla^2u\|_2ds\\
\leq&\varepsilon\int_0^t\|\nabla u_t\|_2^2ds+C\left(\int_0^t\|\nabla u\|_2^2ds\right)^{1/2}\left(\int_0^t\|\nabla^2u\|_2^2ds\right)^{1/2}\sup_{0\leq s\leq t}\|\sqrt\rho u_t\|_2^2\\
\leq&\varepsilon\int_0^t\|\nabla u_t\|_2^2ds+CC_0^{1/2}(\|\nabla u_0\|_2^2+\|\nabla^2d_0\|_2^2)^{1/2}\sup_{0\leq s\leq t}\|\sqrt\rho u_t\|_2^2\\
\leq&\varepsilon\int_0^t\|\nabla u_t\|_2^2ds+C\varepsilon_0^{1/2}\sup_{0\leq s\leq t}\|\sqrt\rho u_t\|_2^2.
\end{align*}
Substituting all these inequalities into (\ref{M.2}) gives
\begin{align}
&\sup_{0\leq s\leq t}(\|\sqrt\rho u_t\|_2^2+\|\nabla d_t\|_2^2)+\int_0^t(\|\nabla u_t\|_2^2+\|d_{tt}\|_2^2+\|\nabla^2d_t\|_2^2)ds\nonumber\\
\leq&(\|\sqrt{\rho_0}u_t(0)\|_2^2+\|\nabla d_t(0)\|_2^2)+C(\|\nabla u_0\|_2^2+\|\nabla^2d_0\|_2^2)^3+C(\|\nabla u_0\|_2^2\nonumber\\
&+\|\nabla^2d_0\|_2^2)^2\int_0^t\|\nabla d_t\|_2^2ds+CC_0(\|\nabla u_0\|_2^2+\|\nabla^2d_0\|_2^2)^2,\label{M.3}
\end{align}
provided $\varepsilon_0$ is small enough. Using Gagliado-Nirenberg inequality  and equation (\ref{1.4}) that
\begin{align}
\|\nabla(\nabla d|^2d-u\cdot\nabla d)\|_2^2
\leq&(\||u|\nabla^2d\|_2^2+\||\nabla u|\nabla d\|_2^2+\|\nabla d\|_6^6+\|\nabla d|\nabla^2d|\|_2^2)\nonumber\\
\leq&C(\|u\|_6^2\|\nabla^2d\|_2\|\nabla^2d\|_6+\|\nabla d\|_6^2\|\nabla u\|_2\|\nabla u\|_6\nonumber\\
&+\|\nabla^2d\|_2^6+\|\nabla d\|_6^2\|\nabla^2d\|_2\|\nabla^2d\|_6^2)\nonumber\\
\leq&C(\|\nabla u\|_2^2\|\nabla^2d\|_2\|\nabla^3d\|_2+\|\nabla^2 d\|_2^2\|\nabla u\|_2\|\nabla^2 u\|_2\nonumber\\
&+\|\nabla^2 d\|_2^2\|\nabla d\|_2^2\|\nabla^3 d\|_2^2+\|\nabla^2 d\|_2^2\|\nabla^2d\|_2\|\nabla^3d\|_2),\label{M.4}
\end{align}
and thus, it follows from Lemma \ref{lemE1} that
\begin{align*}
\int_0^t\|\nabla d_t\|_2^2ds
\leq&C\int_0^t(\|\nabla^3d\|_2^2+\|\nabla^2u\|_2^2+\|\nabla u\|_2^4\|\nabla^2d\|_2^2\\
&+\|\nabla^2d\|_2^4\|\nabla u\|_2^2+\|\nabla^2 d\|_2^2\|\nabla^3 d\|_2^2)ds\\
\leq&C(1+\|\nabla u_0\|_2^2+\|\nabla^2d_0\|_2^2)^3,
\end{align*}
which, combined with (\ref{M.3}), gives
\begin{align}
&\sup_{0\leq s\leq t}(\|\sqrt\rho u_t\|_2^2+\|\nabla d_t\|_2^2)+\int_0^t(\|\nabla u_t\|_2^2+\|d_{tt}\|_2^2+\|\nabla^2d_t\|_2^2)ds\nonumber\\
\leq&(\|\sqrt{\rho_0}u_t(0)\|_2^2+\|\nabla d_t(0)\|_2^2)+C(1+\|\nabla u_0\|_2^2+\|\nabla^2d_0\|_2^2)^4.\label{M.5}
\end{align}
The compatible condition and (\ref{M.4}) implies
$$
\|\sqrt{\rho_0}u_t(0)\|_2^2+\|\nabla d_t(0)\|_2^2\leq C[\|g_0\|_2^2+(1+\|\nabla u_0\|_2^2+\|\nabla d_0\|_{H^2}^2)^2],
$$
from which, recalling (\ref{M.5}), we arrive
\begin{align}
&\sup_{0\leq s\leq t}(\|\sqrt\rho u_t\|_2^2+\|\nabla d_t\|_2^2)+\int_0^t(\|\nabla u_t\|_2^2+\|d_{tt}\|_2^2+\|\nabla^2d_t\|_2^2)ds\nonumber\\
\leq& C[\|g_0\|_2^2+(1+\|\nabla u_0\|_2^2+\|\nabla d_0\|_{H^2}^2)^4].\label{M.5}
\end{align}

By elliptic estimates to Stokes equations and elliptic equations, recalling (\ref{M.4}), it follows from equations (\ref{1.2}) and (\ref{1.4}) that
\begin{align*}
\|\nabla^3d\|_2^2\leq&C\|\nabla d_t\|_2^2+\varepsilon(\|\nabla^2u\|_2^2+\|\nabla^3d\|_2^2)+C(\|\nabla u\|_2^6+\|\nabla^2d\|_2^6)
\end{align*}
and
\begin{align*}
\|\nabla^2u\|_2^2+\|\nabla p\|_2^2\leq&C\|\sqrt\rho u_t\|_2^2+C(\||\nabla d|\nabla^2d\|_2^2+\|\rho u|\nabla u|\|_2^2)\\
\leq&C\|\sqrt\rho u_t\|_2^2+C(\|\nabla d\|_6^2\|\nabla^2d\|_2\|\nabla^3d\|_2+\|u\|_6^2\|\nabla u\|_2\|\nabla^2u\|_2)\\
\leq&C\|\sqrt\rho u_t\|_2^2+\varepsilon(\|\nabla^2u\|_2^2+\|\nabla^3d\|_2^2)+C(\|\nabla^2d\|_2^6+\|\nabla u\|_2^6),
\end{align*}
and thus
$$
\|\nabla^2u\|_2^2+\|\nabla p\|_2^2+\|\nabla^3d\|_2^2\leq C(\|\sqrt\rho u_t\|_2^2+\|\nabla d_t\|_2^2)+C(\|\nabla^2d\|_2^6+\|\nabla u\|_2^6).
$$
On account of this inequality, it follows from (\ref{M.5}) and Lemma \ref{lemE1} that
\begin{align*}
&\sup_{0\leq s\leq t}(\|\sqrt\rho u_t\|_2^2+\|\nabla d_t\|_2^2+\|\nabla^2u\|_2^2+\|\nabla p\|_2^2+\|\nabla^3d\|_2^2)\\
+&\int_0^t(\|\nabla u_t\|_2^2+\|d_{tt}\|_2^2+\|\nabla^2d_t\|_2^2)ds\leq C[\|g_0\|_2^2+(1+\|\nabla u_0\|_2^2+\|\nabla d_0\|_{H^2}^2)^4],
\end{align*}
completing the proof.
\end{proof}

\begin{proposition}\label{propapriori}
Let $(\rho, u, d, p)$ be a strong solution to the system (\ref{1.1})--(\ref{1.6}) on $Q_T$, and set
$$
E(t)=\sup_{0\leq s\leq t}(\|\nabla\rho\|_2^2+\|\rho_t\|_2^2+\|u\|_{H^2}^2+\|\nabla d\|_{H^2}^2+\|\nabla p\|_2^2)+\int_0^t(\|\nabla u_t\|_2^2+\|\nabla^2d_t\|_2^2)ds.
$$
Then there is a positive constant $\varepsilon_0$ depending only on $\overline\rho$ and $\Omega$, such that
$$
E(t)\leq C(\|\rho_0\|_{H^1}, \|u_0\|_{H^2},\|\nabla d_0\|_{H^2}, \|g_0\|_2, \overline\rho, \Omega)
$$
for any $t\in(0, T)$, provided $\|\sqrt{\rho_0}u_0\|_2^2+\|\nabla d_0\|_2^2\leq\varepsilon_0$ if $N=2$, and $(\|\sqrt{\rho_0}u_0\|_2^2+\|\nabla d_0\|_2^2)(\|\nabla u_0\|_2^2+\|\nabla^2d_0\|_2^2)\leq\varepsilon_0$ if $N=3$.
\end{proposition}

\begin{proof}
Using Sobolev embedding inequality and Gagliado-Nirenberg inequlity and applying elliptic estimates on Stokes equations, it follows from Lemma \ref{lemE1} and Lemma \ref{lemE2} that
\begin{align*}
\|\nabla^2u\|_4^2\leq& C(\|\rho u_t\|_4^2+\|\rho u|\nabla u|\|_4^2+\|\nabla d|\nabla^2d|\|_4^2)\\
\leq&C(\|\nabla u_t\|_2^2+\|u\|_8^2\|\nabla u\|_8^2+\|\nabla d\|_8^2\|\nabla^2d\|_8^2)\\
\leq&C(\|\nabla u_t\|_2^2+\|\nabla u\|_2^2\|\nabla^2u\|_2^2+\|\nabla^2d\|_2^2\|\nabla^3d\|_2^2)\\
\leq&C(\|\nabla u_t\|_2^2+1),
\end{align*}
if $N=2$, and
\begin{align*}
\|\nabla^2u\|_4^2\leq& C(\|\rho u_t\|_4^2+\|\rho u|\nabla u|\|_4^2+\|\nabla d|\nabla^2d|\|_4^2)\\
\leq&C(\|\nabla u_t\|_2^2+\|u\|_\infty^2\|\nabla u\|_4^2+\|\nabla d\|_{12}^2\|\nabla^2d\|_6^2)\\
\leq&C(\|\nabla u_t\|_2^2+\|\nabla^2 u\|_2^2\|\nabla u\|_2^{14/13}\|\nabla^2u\|_4^{12/13}+\|\nabla^3d\|_2^2\|\nabla^3d\|_2^2)\\
\leq&\frac{1}{2}\|\nabla^2u\|_2^2+C\|\nabla u_t\|_2^2+C,
\end{align*}
if $N=3$. As a consequence, it follows from Lemma \ref{lemE2} that $\int_0^t\|\nabla^2u\|_4^2ds\leq C.$ On account of this inequality, it follows from Lemma \ref{est on transport} that $\|\nabla\rho(t)\|_2^2\leq C\|\nabla\rho_0\|_2^2$. Using equation (\ref{1.1}), by Sobolev embedding inequality and Lemma \ref{lemE2}, we deduce
$$
\|\rho_t\|_2^2\leq C\|u\|_\infty^2\|\nabla\rho\|_2^2\leq C\|\nabla^2u\|_2^2\|\nabla\rho\|_2^2\leq C.
$$
The proof is complete.
\end{proof}

\section{Proof of the main theorem}\label{sec4}

After obtaining the a priori estimates stated in the previous section, we can now give the proof of our main results.

\begin{proof}
We divide the proof into several steps. Let $\varepsilon_0$ be the constant stated in Proposition \ref{propapriori} and suppose that
\begin{eqnarray*}
&&\|\sqrt{\rho_0}u_0\|_2^2+\|\nabla d_0\|_2^2<\varepsilon_0\mbox{ if }N=2,\\
&&(\|\sqrt{\rho_0}u_0\|_2^2+\|\nabla d_0\|_2^2)(\|\nabla u_0\|_2^2+\|\nabla^2d_0\|_2^2)<\varepsilon_0 \mbox{ if }N=3.
\end{eqnarray*}

Step 1. Global existence of strong solutions without vacuum. Suppose that there is a positive constant $\underline\rho$ such that $\rho_0(x)\geq\underline\rho$ for all $x\in\Omega$. By Lemma \ref{local existence}, there is
a unique strong solution to system (\ref{1.1})--(\ref{1.6}) in $Q_T$. Let $T_*$ be the maximal existence time. We claim that $T_*=\infty$. If it's not the case,
then $0<T_*<\infty$. By Proposition \ref{propapriori}, there is a positive constant $K$, such that
\begin{align}
\sup_{0\leq t\leq T}&(\|\nabla d\|_{H^2}^2+\|\nabla p\|_{2}^2+\|u\|_{H^2}^2+\|\nabla\rho \|_2^2+\|\rho_t\|_2^2)+\int_0^T(\|d_t\|_{H^2}^2+\|u_t\|_{H^1}^2)ds\leq K
\label{4.1}
\end{align}
for some suitable constant $K$ depending only on $\|g_0\|_2$, $\|u_0\|_{H^2}$, $\|\nabla d_0\|_{H^2}$, $\bar\rho$ and $\Omega$.
Since $\rho_0(x)\geq\underline\rho>0$ for all $x\in\Omega$, Lemma \ref{est on transport} gives $\rho(x,t)\geq\underline\rho$ for all $(x,t)\in Q_{T_*}$, and thus
$$
\sup_{0< t<T_*}\|\rho^{-1/2}(-\Delta u-\textmd{div}(\nabla d\odot\nabla d))\|_2^2\leq C
$$
for some constant $C$. Using Lemma \ref{local existence}, we can extend the strong solution to $(0,T^*)$ with some $T^*>T_*$, which contradicts to the definition of $T_*$. This contradiction implies that $T_*=\infty$, and thus we obtain the global strong solutions for system (\ref{1.1})--(\ref{1.6}) with initial density being away from vacuum. This completes the proof of step 1.

Step 2. Global existence of strong solutions with vacuum. We prove it by approximation. For any $j=1,2,\cdots,$ define $\rho_0^{j}=\rho_0+\frac{1}{j}$ and
$$
g_0^j=(\rho_0^j)^{-1/2}(-\Delta u_0-\nabla p_0-\textmd{div}(\nabla d_0\odot\nabla d_0)),
$$
then the compatible condition holds true
$$
-\Delta u_0-\nabla p_0-\textmd{div}(\nabla d_0\odot\nabla d_0)=\sqrt{\rho_0^j}g_0^j
$$
in $\Omega$.
It's easy to see that $|g_0^j|\leq|g_0|$ in $\Omega$ and thus $\|g_0^j\|_2^2\leq\|g_0\|_2^2$ for all $j$.
Then for $j$ large enough it has
\begin{eqnarray*}
&&\|\sqrt{\rho_0^j}u_0\|_2^2+\|\nabla d_0\|_2^2<\varepsilon_0\mbox{ if }N=2,\\
&&(\|\sqrt{\rho_0^j}u_0\|_2^2+\|\nabla d_0\|_2^2)(\|\nabla u_0\|_2^2+\|\nabla^2d_0\|_2^2)<\varepsilon_0 \mbox{ if }N=3.
\end{eqnarray*}

Using the result we have proven in step 1, for each $j$ large enough, there is a global strong solution
$(\rho^j,u^j,p^j,d^j)$ to system (\ref{1.1})--(\ref{1.6}) with initial data $\rho^j(0)=\rho_0^j$ and $u^j(0)=u_0$. Moreover, by Proposition \ref{propapriori}, $(\rho^j,u^j,p^j,d^j)$ satisfies the following estimates
\begin{align*}
\sup_{0\leq t\leq T}&(\|\nabla d\|_{H^2}^2+\|\nabla p\|_{2}^2+\|u\|_{H^2}^2+\|\nabla\rho \|_2^2+\|\rho_t\|_2^2)+\int_0^T(\|d_t\|_{H^2}^2+\|u_t\|_{H^1}^2)ds\leq K.
\end{align*}
By the aid of these uniform estimates with respect to $j$, we can follow the standard convergence approach to obtain the global strong solution to system (\ref{1.1})--(\ref{1.6}) satisfying the regularities and estimates stated in Theorem \ref{theorem1}. Since it's quite a standard approach, we omit it here.

Step 3. Uniqueness of global strong solutions.
Let $(\rho,u,p,d)$ and $(\hat\rho,\hat u,\hat p,\hat d)$ be two global strong solutions to system (\ref{1.1}) to (\ref{1.4}) satisfying the regularities and
estimates stated in Theorem \ref{theorem1}. We only consider the three dimensional case, the two dimensional case can be dealt with in the similar way. Using the identity
\begin{equation*}
(\rho-\hat\rho)_t+u\cdot\nabla(\rho-\hat\rho)=(\hat u-u)\cdot\nabla\hat\rho,
\end{equation*}
we can deduce that
\begin{equation*}
(|\rho-\hat\rho|^{3/2})_t+u\cdot\nabla(|\rho-\hat\rho|^{3/2})\leq\frac{3}{2}|u-\hat u||\nabla\hat\rho||\rho-\hat\rho|^{1/2}.
\end{equation*}
Integrating over $\Omega$, we deduce from Sobolev embedding inequality that
\begin{align*}
\frac{d}{dt}\int_\Omega|\rho-\hat\rho|^{3/2}dx\leq&\frac{3}{2}\int_\Omega|\hat u-u||\nabla\hat\rho||\rho-\hat\rho|^{1/2}dx\\
\leq&\frac{3}{2}\|u-\hat u\|_6\|\nabla\hat\rho\|_2\|\rho-\hat\rho\|_{3/2}^{1/2}\leq C\|\nabla(u-\hat u)\|_2\|\nabla\hat\rho\|_2\|\rho-\hat\rho\|_{3/2}^{1/2}
\end{align*}
and thus
\begin{align}
\frac{d}{dt}\|\rho-\hat\rho\|_{3/2}^2=&\frac{4}{3}\|\rho-\hat\rho\|_{3/2}^{1/2}\frac{d}{dt}\|\rho-\hat\rho\|_{3/2}^{3/2}\nonumber\\
\leq& C\|\nabla(\hat u-u)\|_2\|\nabla\hat\rho\|_2\|\rho-\hat\rho\|_{3/2}\nonumber\\
\leq&\frac{1}{4}\|\nabla(u-\hat u)\|_2^2+\alpha(t)\|\rho-\hat\rho\|_{3/2}^2\label{49}
\end{align}
for some nonnegative function $\alpha(t)\in L^1(0,T)$ for any $T>0$.
Using equation (\ref{1.2}), we deduce
\begin{align*}
&\rho(u-\hat u)_t+\rho u\cdot\nabla(u-\hat u)-\Delta(u-\hat u)+\nabla(p-\hat p)\nonumber\\
=&\Delta(\hat d-d)\cdot\nabla\hat d+\Delta d\cdot\nabla(\hat d-d)-(\rho-\hat\rho)(\hat u_t+\hat u\cdot\nabla\hat u)-\rho(u-\hat u)\cdot\nabla\hat u.
\end{align*}
Multiplying the above equation by $u-\hat u$ and integration by parts yields
\begin{align*}
&\frac{1}{2}\frac{d}{dt}\int_\Omega\rho|u-\hat u|^2dx+\int_\Omega|\nabla(u-\hat u)|^2dx\\
\leq&\int_\Omega(|\nabla(d-\hat d)||\nabla\hat d||\nabla(u-\hat u)|+|\nabla(d-\hat d)||\nabla^2\hat d||u-\hat u|+|\Delta d||\nabla(\hat d-d)||u-\hat u|\\
&+|\rho-\hat\rho||\hat u_t+u\cdot\nabla\hat u||u-\hat u|+\rho|\nabla\hat u||u-\hat u|^2)dx,
\end{align*}
consequently, it follows from H\"older inequality, Sobolev embedding inequality and Young inequality that
\begin{align*}
&\frac{d}{dt}\|\sqrt\rho(u-\hat u)\|_2^2+\|\nabla(u-\hat u)\|_2^2\\
\leq&\frac{1}{4}\|\nabla(u-\hat u)\|_2^2+C\|\nabla\hat d\|_\infty^2\|\nabla(d-\hat d)\|_2^2\\
&+C\Big(\|\nabla^2\hat d\|_3\|\nabla(\hat d-d)\|_2\|u-\hat u\|_6+\|\Delta d\|_3\|\nabla(d-\hat d)\|_2\|u-\hat u\|_6\\
&+\|\rho-\hat\rho\|_{3/2}\|\hat u_t+u\cdot\nabla\hat u\|_6\|u-\hat u\|_6+\|\nabla\hat u\|_\infty\|\sqrt\rho(u-\hat u)\|_2^2\Big)\\
\leq&\frac{1}{2}\|\nabla(u-\hat u)\|_2^2+C\Big[(\|\nabla\hat d\|_\infty^2+\|\nabla^2\hat d\|_3^2+\|\Delta d\|_3^2)\|\nabla(d-\hat d)\|_2^2\\
&+\|\hat u_t+u\cdot\nabla\hat u\|_6^2\|\rho-\hat\rho\|_{3/2}^2+\|\nabla\hat u\|_\infty\|\sqrt\rho(u-\hat u)\|_2^2\Big],
\end{align*}
from which we deduce
\begin{align*}
&\frac{d}{dt}\|\sqrt\rho(u-\hat u)\|_2^2+\|\nabla(u-\hat u)\|_2^2\\
\leq&\beta(t)(\|\nabla(d-\hat d)\|_2^2
+\|\rho-\hat\rho\|_{3/2}^2+\|\sqrt\rho(u-\hat u)\|_2^2),
\end{align*}
for some nonnegative function $\beta(t)\in L^1(0,T)$ for any $T>0$.
Combining the above inequality with (\ref{49}), we obtain
\begin{align}
&\frac{d}{dt}\Big(\|\rho-\hat\rho\|_{3/2}^2+\|\sqrt\rho(u-\hat u)\|_2^2\Big)+\|\nabla(u-\hat u)\|_2^2 \nonumber\\
\leq&\gamma(t)(\|\nabla(d-\hat d)\|_2^2+\|\rho-\hat\rho\|_{3/2}^2+\|\sqrt\rho(u-\hat u)\|_2^2),\label{50}
\end{align}
for some nonnegative function $\gamma(t)\in L^1(0,T)$ for any $T>0$. Using (\ref{1.4}), it has
\begin{align*}
&(d-\hat d)_t+u\cdot\nabla(d-\hat d)-\Delta(d-\hat d)\nonumber\\
=&\nabla(d-\hat d):\nabla(d+\hat d)d+|\nabla\hat d|^2(d-\hat d)-(u-\hat u)\cdot\nabla\hat d.
\end{align*}
Multiply the above equation by $-\Delta(d-\hat d)$ and integration by parts, we obtain
\begin{align*}
&\frac{1}{2}\frac{d}{dt}\int_\Omega|\nabla(d-\hat d)|^2dx+\int_\Omega|\Delta(d-\hat d)|^2dx\\
=&\int_\Omega\Big[u\cdot\nabla(d-\hat d)+(u-\hat u)\cdot\nabla\hat d-\nabla(d-\hat d):\nabla(d+\hat d)d-|\nabla\hat d|^2(d-\hat d)\Big]\cdot\Delta(d-\hat d)dx\\
\leq&\frac{1}{2}\int_\Omega|\Delta(d-\hat d)|^2dx+\int_\Omega\Big[(|u|^2+|\nabla(d+\hat d)|^2)|\nabla(d-\hat d)|^2+|\nabla \hat d|^4|d-\hat d|^2+|\nabla\hat d
|^2|u-\hat u|^2\Big]dx.
\end{align*}
It follows from Poinc\'are inequality and the above inequality that
\begin{align*}
\frac{d}{dt}\|\nabla(d-\hat d)\|_2^2+\|\Delta(d-\hat d)\|_2^2\leq\zeta(t)(\|\nabla(d-\hat d)\|_2^2+\|\nabla(u-\hat u)\|_2^2),
\end{align*}
for some nonnegative function $\zeta(t)\in L^1(0,T)$ for any $T>0$. Combining the above inequality with (\ref{50}), we have
\begin{align*}
&\frac{d}{dt}\Big(\|\rho-\hat\rho\|_{3/2}^2+\|\sqrt\rho(u-\hat u)\|_2^2+\|\nabla(d-\hat d)\|_2^2\Big)+(\|\nabla(u-\hat u)\|_2^2+\|\Delta(d-\hat d)\|_2^2) \nonumber\\
\leq&\xi(t)(\|\rho-\hat\rho\|_{3/2}^2+\|\sqrt\rho(u-\hat u)\|_2^2+\|\nabla(d-\hat d)\|_2^2),
\end{align*}
for some nonnegative function $\xi(t)\in L^1(0,T)$ for any $T>0$.
Recalling that $(\rho-\hat\rho,u-\hat u,d-\hat d)|_{t=0}=(0,0,0)$, it follows from Gronwall's inequality that $(\rho,u,d)=(\hat\rho,\hat u,\hat d)$, and consequently
$p=\hat p$, completing the proof.
\end{proof}

\end{document}